\documentclass[10pt]{article}

\usepackage{amsmath, amssymb ,amsthm, amsfonts, amsgen}

\numberwithin{equation}{section}

\newcommand{\ds}{\displaystyle}
\newcommand{\Nb}{{\mathbb{N}}}
\newcommand{\Rb}{{\mathbb{R}}}
\newcommand{\Zb}{{\mathbb{Z}}}
\newcommand{\A}{{\mathcal{A}}}
\newcommand{\R}{{\mathcal{R}}}
\newcommand{\C}{{\mathcal{C}}}
\newcommand{\E}{{\mathcal{E}}}
\newcommand{\LL}{{\mathcal{L}}}
\newcommand{\HH}{{\mathcal{H}}}
\newcommand{\M}{{\mathcal{M}}}

\def\dist{\text{dist}}

\makeatletter
\def\rightharpoonupfill@{\arrowfill@\relbar\relbar\rightharpoonup}
\newcommand{\xrightharpoonup}[2][]{\ext@arrow
0359\rightharpoonupfill@{#1}{#2}} \makeatother

\newcommand{\res}{\mathop{\hbox{\vrule height 7pt width .5pt depth 0pt
\vrule height .5pt width 6pt depth 0pt}}\nolimits}

\let\G=\Gamma
\let\a=\alpha
\let\b=\beta
\let\d=\delta
\let\g=\gamma
\def\e{{\varepsilon}}
\def\O{{\Omega}}
\def\o{{\omega}}
\def\x{{\times}}
\def\ovb{\overline{b}}

\newtheorem{Theorem}{Theorem}[section]
\newtheorem{Lemma}[Theorem]{Lemma}
\newtheorem{Proposition}[Theorem]{Proposition}

\newtheorem{Remark}[Theorem]{Remark}

\newtheorem{Example}[Theorem]{Example}

\newcommand\keywordsname{Key words}
\newcommand\AMSname{AMS subject classifications}

\begin{document}

\author{Jean-Fran\c cois Babadjian\footnote{Laboratoire Jean Kuntzmann, Universit\'e Joseph
Fourier, BP 53, 38041 Grenoble Cedex 9 (FRANCE).\newline Email:
\texttt{babadjia@imag.fr}}\and Elvira Zappale\footnote{DIIMA,
Universit\`a degli Studi di Salerno, via Ponte Don Melillo, 84084
Fisciano (SA) (ITALY). \newline Email:
\texttt{zappale@diima.unisa.it}} \and Hamdi
Zorgati\footnote{D\'epartement de Math\'ematiques, Campus
Universitaire, Universit\'e Tunis El Manar 2092 (TUNISIA). \newline
Email: \texttt{hamdi.zorgati@fst.rnu.tn}}}

\title{Dimensional reduction for energies with linear growth involving the bending moment}

\date{}

\maketitle
\begin{abstract}
A $\Gamma$-convergence analysis is used to perform a 3D-2D dimension
reduction of variational problems with linear growth. The adopted
scaling gives rise to a nonlinear membrane model which, because of
the presence of higher order external loadings inducing a bending
moment, may depend on the average in the transverse direction of a
Cosserat vector field, as well as on the deformation of the
mid-plane. The assumption of linear growth on the energy leads to an
asymptotic analysis in the spaces of measures and of functions with
bounded variation.

\medskip

\begin{center}
{\bf R\'esum\'e}
\end{center}

Une analyse variationnelle par $\Gamma$-convergence est utilis\'ee
pour \'etudier un probl\`eme de r\'eduction de dimension 3D-2D pour
des \'energies \`a croissance lin\'eaire. La mise \`a l'\'echelle
donne lieu \`a un mod\`ele effectif de membrane qui, en vertu de la
pr\'esence de forces ext\'erieures engendrant un moment
fl\'echissant, d\'epend de la moyenne dans la direction transverse
du vecteur de Cosserat ainsi que de la d\'eformation de la surface
moyenne. L'hypoth\`ese de croissance lin\'eaire n\'ecessite une
analyse asymptotique dans les espaces de mesures et de fonctions \`a
variation born\'ee.

\medskip


\medskip

\end{abstract}

%\tableofcontents

\section{Introduction}

%%%%%%%%%%%%%%%%%%%%%%%%%%%%%%%%%%%%%%%%%%%%%%%%%%%%%%%%%%%%%%%%

\noindent In solid mechanics, the equilibrium state of a body may be
described by an energy minimization problem. When we deal with very
thin structures, {\it i.e.}, structures whose thickness is much
smaller than the other dimensions, it is convenient to consider a
lower-dimensional model describing the behavior of the minimizing
sequences when the thickness goes to zero in the thin direction. The
knowledge of these asymptotic models may be useful, for example, in
numerical implementation since it gives less cost of time of
calculus.

In the seminal paper \cite{LR}, the authors derived a nonlinear
membrane model from three dimensional nonlinear elasticity, for
energies having a polynomial growth of order $p>1$. They computed
the $\G$-limit in the Sobolev space $W^{1,p}$ of the elastic energy
without any convexity condition. A general integral representation
result has been later established in \cite{BFF} where applications
to heterogeneous bodies in the transverse direction, homogenization
and optimal design problems are given. The case of completely
heterogeneous materials has been carried out in \cite{BF}. We also
refer to \cite{B0,B,BouFonLeoMas,BrF} for the study of fractured
thin films in the space $SBV^p$ of special functions with bounded
variation. In \cite{BFM}, a richer model has been proposed
introducing higher order surface loadings. It leads to bending
moment effects enhanced, in the asymptotic model, through the
explicit dependence on the average in the transverse direction of a
Cosserat vector field. A generalization to heterogeneous media has
been given in \cite{BF} and an abstract integral representation
result in $W^{1,p}$ (and also $SBV^p$) has been proved in \cite{B}.

In this paper, we seek to derive a two-dimensional nonlinear
membrane model from three-dimensional nonlinear elasticity involving
a bulk energy with linear growth $(p=1)$. As in  \cite{B,BF,BFM} we
allow the presence of higher order surface loadings inducing a
bending moment. Due to the linear growth of the energy, the limit
model depends on a two-dimensional deformation which belongs to the
space $BV$ of functions with bounded variation, and on a Cosserat
vector which is a Radon measure. Note that dimensional reduction
problems for energies with linear growth have also been studied in
\cite{BrF} for cracked thin films. In this case, the 3D-energy which
is the sum of a bulk and a surface term penalizing the presence of
the cracks, is defined in the space $SBV$.

Let us consider $\o$ a bounded open subset of $\Rb^2$ with Lipschitz
boundary and set $\O_\e :=\o \times (-\e/2,\e/2)$. We assume that
$\O_\e$ stands for the reference configuration of a homogeneous
nonlinear elastic thin film whose stored energy density is given by
the Borel function $W:\Rb^{3 \times 3} \to [0,+\infty)$. Our first
main assumption is that $W$ satisfies some linear growth and
coercivity conditions, {\it i.e.}, there exist $0<\b' \leq \b <
+\infty$ such that
$$\b'|\xi| \leq W(\xi) \leq \b(1+|\xi|)\quad \text{ for every }\xi
\in \Rb^{3 \times 3}.$$ To fix ideas, suppose that the body is
clamped on the lateral boundary $\G_\e:=\partial \o \times
(-\e/2,\e/2)$, and that the sections $\Sigma_\e:=\o \times \{\pm
\e/2\}$ are subjected to $\e$-dependent external loadings
$g(\e):\Sigma_\e\to \Rb^3$. Assume further that the material is
submitted to the action of a body load $f(\e):\O_\e \to \Rb^3$ so
that the total energy of the system, which is given by the
difference between the elastic energy and the work of external
forces, is
$$\E(\e)(v):=\int_{\O_\e} W(\nabla v)\, dx - \int_{\O_\e} f(\e)
\cdot v\, dx - \int_{\Sigma_\e} g(\e) \cdot v \, d\HH^2,$$ for any
kinematically admissible deformation field $v : \O_\e \to \Rb^3$
satisfying $v(x)=x$ on $\G_\e$.

Thanks to the growth condition satisfied by $W$, we have -- at this
stage -- a good functional setting if we assume any kinematically
admissible deformation fields to belong to the space $\mathcal
V(\e):= \{\varphi \in W^{1,1}(\O_\e;\Rb^3):\; T\varphi=x \text{ on
}\G_\e\}$, where $T\varphi$ denotes the trace of $\varphi$ on the
lateral boundary $\G_\e$. The problem consists in finding
equilibrium states of this body, in other words finding minimizers
of the functional $\E(\e)$ over the space $\mathcal V(\e)$.

As explained before, a natural question which arises is the study of
the asymptotic behavior of such energies as well as their (eventual)
minimizers as the thickness parameter $\e$ tends to zero. This will
be performed by means of a $\G$-convergence analysis (see {\it e.g.}
\cite{Br,DM} for a comprehensive treatment). It is now usual to
rescale the problem on a fixed domain $\O:=\o \times I$ of unit
thickness, where $I:=(-1/2,1/2)$. Similarly set $\Sigma:=\o \times
\{\pm 1/2\}$ and $\G:=\partial \o \times I$. Denoting by
$x_\a:=(x_1,x_2)$ the in-plane variable, we define
$g_\e(x_\a,x_3):=g(\e)(x_\a,\e x_3)$, $f_\e
(x_\a,x_3):=f(\e)(x_\a,\e x_3)$, $u(x_\a,x_3):=v(x_\a,\e x_3)$ and
$\E_\e(u)=\E(\e)(v)/\e$ so that
$$\E_\e(u)=\int_\O W\left( \nabla_\a u\Big|\frac{1}{\e} \nabla_3 u \right) dx
- \int_\O f_\e \cdot u\, dx - \int_\Sigma g_\e \cdot u\, d\HH^2.$$
Note that since we divided the total energy by $\e$, we expect to
get a term of order $\e$ in the limit model which corresponds,
according to the formal asymptotic expansion performed in
\cite{FRS}, to a membrane energy which only accounts for stretching
effects.

Provided the rescaled external forces $f_\e$ and $g_\e$ have an
appropriate order of magnitude (which will be discussed later), it
follows from the growth condition satisfied by $W$ and some
Poincar\'e type inequality, that minimizing sequences $\{u_\e\}$
with finite total energy will be bounded in $W^{1,1}(\O;\Rb^3)$.
Actually, the ``scaled'' gradient of $u_\e$, {\it i.e.},
$\{(\nabla_\a u_\e|(1/\e)\nabla_3 u_\e)\}$, will be uniformly
bounded in $L^1(\O;\Rb^{3 \times 3})$. However, because of the lack
of reflexibility of $W^{1,1}(\O;\Rb^3)$, such minimizing sequences
will only be relatively compact in the larger space $BV(\O;\Rb^3)$
of functions with bounded variation. Denoting by $u$ any weak* limit
in $BV(\O;\Rb^3)$ of the sequence $\{u_\e\}$, it turns out that the
only interesting deformations (according to this scaling) will
necessary satisfy $D_3 u=0$ in the sense of distributions. Hence $u$
(can be identified to a function which) belongs to $ BV(\o;\Rb^3)$
and we expect a ($\G$-)limit model depending on such deformations.

Our second main assumption is that the (rescaled) surface load can
be written as $g_\e=g_0/\e+g_1$. It follows from \cite[Remark
2.3.2]{FRS} that, denoting by $g_i^\pm$ ($i=0$ or $1$) the trace of
$g_i$ on $\o \times \{\pm 1/2\}$, the condition $g_0^+ + g_0^-=0$
must hold. The physical interpretation of this property is that a
plate of thickness $\e$ cannot support a non vanishing resultant
surface load as $\e \to 0$. Assume also for simplicity that
$f_\e=f$. If $\{u_\e\} \subset W^{1,1}(\O;\Rb^3)$ is a minimizing
sequence as above, the work of external forces has the following
form
\begin{eqnarray*}\mathcal F_\e(u_\e) & := & \int_\O f\cdot
u_\e\, dx + \int_\Sigma g_1 \cdot u_\e \, d\HH^2 +  \int_\o g_0^+
\cdot \left(\frac{u_\e(\cdot,+1/2) -
u_\e(\cdot,-1/2)}{\e}\right)dx_\a\\
& = & \int_\O f\cdot u_\e\, dx + \int_\Sigma g_1 \cdot u_\e \,
d\HH^2 +  \int_\o g_0^+ \cdot \left(\frac{1}{\e}\int_I \nabla_3
u_\e(\cdot,y_3)\, dy_3 \right)dx_\a.
\end{eqnarray*}
Let $u \in BV(\o;\Rb^3)$ be an accumulation point of $\{u_\e\}$ and
$\ovb \in \M(\o;\Rb^3)$ be a weak* limit in the space of Radon
measures of the sequence $$\left\{\frac{1}{\e} \int_I \nabla_3
u_\e(\cdot,y_3)\, dy_3\right\}$$ which does always exist up to a
subsequence. Taking the limit as $\e \to 0$ in the work of external
forces, and denoting $\overline f(x_\a):=\int_If(x_\a,x_3)\, dx_3$
yields
$$\mathcal F_\e(u_\e)\to \mathcal F(u,\ovb):= \int_\o
\big(\overline f+g_1^++g_1^-\big) \cdot u\, dx_\a + \int_\o g_0^+
\cdot d\ovb,$$ provided $f$, $g_1$ and $g_0$ are regular enough,
{\it e.g.}, $f \in L^\infty(\O;\Rb^3)$, $g_1^\pm \in
L^\infty(\o;\Rb^3)$ and $g_0^+ \in \C_0(\o;\Rb^3)$. The presence of
this higher order surface load implies the apparition in the limit
of the average in the transverse direction of the Cosserat measure
$\ovb$ which stands for bending moment effects (see
\cite{B,BF,BFM}). Hence we seek a richer $\G$-limit depending on
both $u$ and $\ovb$. Note that in general, $u$ and $\ovb$ are
completely independent macroscopic entities, and as a matter of
fact, it may happen that the measures $D_\a u$ and $\ovb$ are
mutually singular (see Example \ref{ex}).\\

The following theorem is the main result of this work and it
describes the behavior of the elastic energy as $\e \to 0$.  We
refer to section 2 for the notations used in the statement.
\begin{Theorem}\label{BZZ}
Let $\o \subset \Rb^2$ be a bounded open set and $W:\Rb^{3 \times 3}
\to [0,+\infty)$ be a Borel function satisfying
\begin{itemize}
\item[$(H_1)$] there exist $0 < \b' \leq \b < +\infty$ such that
$$\b'|\xi| \leq W(\xi) \leq \b(1+|\xi|) \quad \text{ for all }\xi
\in \Rb^{3 \times 3};$$
\item[$(H_2)$] there exist $C>0$ and $r \in (0,1)$ such that
$$|W^\infty(\xi) - W(\xi)| \leq C(1+|\xi|^{1-r}) \quad \text{ for all }\xi
\in \Rb^{3 \times 3},$$ where $W^\infty$ is the recession function
of $W$.
\end{itemize}
Then, for every $(u,\ovb) \in BV(\O;\Rb^3) \times \M(\o;\Rb^3)$, the
sequence of functionals
$$J_\e(u,\ovb):=\left\{
\begin{array}{ll}
\ds \int_\O W\left(\nabla_\a u \Big|\frac{1}{\e} \nabla_3 u\right)dx
& \text{ if } \left\{
\begin{array}{l}
u \in W^{1,1}(\O;\Rb^3),\\
\ovb=\frac{1}{\e}\int_I \nabla_3 u(\cdot,x_3)\, dx_3,
\end{array}
\right.\\[0.5cm] +\infty & \text{ otherwise,}
\end{array}
\right.$$ $\G$-converges for the weak* topology of $BV(\O;\Rb^3)
\times \M(\o;\Rb^3)$ to
\begin{equation}\label{gammarep1}E(u,\ovb):=\left\{
\begin{array}{ll}
\begin{array}{l}\ds \int_\o \mathcal Q^* W\left(\nabla_\a
u\Big|\frac{d\ovb}{d\LL^2}\right) dx_\a\\[0.4cm]
\ds \hspace{1cm} + \int_{J_u} (\mathcal Q^\ast
W)^\infty\left((u^+-u^-)\otimes\nu_u,\frac{d\ovb}{d\HH^1\res\, J_u}
\right)
d\HH^1\\[0.4cm]
\ds \hspace{1cm} + \int_\o (\mathcal Q^* W)^\infty\left(\frac{dD_\a
u}{d|D_\a^c
u|} \Big|\frac{d\ovb}{d|D_\a^c u|}\right) d|D_\a^c u|\\[0.4cm]
\ds \hspace{1cm} + \int_\o (\mathcal Q^* W)^\infty\left(0
\Big|\frac{d\ovb}{d|\ovb^\sigma|}\right) d|\ovb^\sigma| \end{array}
& \text{
if }u\in BV(\o;\Rb^3),\\[0.5cm]
+\infty & \text{ otherwise},
\end{array}
\right.
\end{equation} where
\begin{eqnarray*}{\cal Q}^\ast W(\overline \xi|b) & := &
\inf_{\lambda, \, \varphi} \left\{\int_{Q' \times I} W(\overline \xi
+ \nabla_\a \varphi|\lambda \nabla_3 \varphi)\,dx: \lambda >0,\,
\varphi \in W^{1,1}(Q' \times I;\Rb^3),\right.\nonumber\\
&&\left. \varphi(\cdot, x_3) \hbox{ is } Q'\hbox{-periodic for }
\LL^1 \hbox{-a.e. } x_3 \in I, \lambda \int_{Q' \times I}\nabla_3
\varphi(y)\, dy =b\right\},\end{eqnarray*} for all $(\overline
\xi|b) \in \Rb^{3 \times 2} \times \Rb^3$, $(\mathcal Q^* W)^\infty$
is the recession function of $\mathcal Q^* W$ and $\ovb^\sigma$ is
the singular part of $\ovb$ with respect to $|D_\a u|$ according to
the Besicovitch Decomposition Theorem.
\end{Theorem}

\begin{Remark}{\rm
The fact that $E$ is the $\G$-limit of the family $\{J_\e\}$ for the
weak* topology of $BV(\O;\Rb^3) \times \M(\o;\Rb^3)$ means that for
every $(u,\ovb) \in BV(\o;\Rb^3) \times \M(\o;\Rb^3)$ and every
sequence $\{\e_j\} \searrow 0^+$, then:
\begin{enumerate}
\item[(i)] for any sequence $\{u_j\} \subset W^{1,1}(\O;\Rb^3)$ such
that $u_j \xrightharpoonup[]{*} u$ in $BV(\O;\Rb^3)$ and
$\frac{1}{\e_j}\int_I \nabla_3 u_j(\cdot,x_3)\, dx_3
\xrightharpoonup[]{*} \ovb$ in $\M(\o;\Rb^3)$,
$$E(u,\ovb) \leq \liminf_{j \to +\infty}\int_\O W\left(\nabla_\a u_j \Big|\frac{1}{\e_j}\nabla_3 u_j
\right)dx;$$
\item[(ii)] there exists a sequence $\{\bar u_j\} \subset W^{1,1}(\O;\Rb^3)$ such
that $\bar u_j \xrightharpoonup[]{*} u$ in $BV(\O;\Rb^3)$,
$\frac{1}{\e_j}\int_I \nabla_3 \bar u_j(\cdot,x_3)\, dx_3
\xrightharpoonup[]{*} \ovb$ in $\M(\o;\Rb^3)$, and
$$E(u,\ovb) = \lim_{j \to +\infty}\int_\O W\left(\nabla_\a \bar u_j \Big|\frac{1}{\e_j}\nabla_3 \bar u_j
\right)dx.$$
\end{enumerate}
}\end{Remark}

The strategy used to prove Theorem \ref{BZZ} is based on the blow-up
method introduced in \cite{FM0,FM} for the study of the relaxation
of integral functionals with linear growth. It rests on a
localization of the energy around convenient Lebesgue points, and
uses fine properties of measures and $BV$ functions at these points.
We adapt here this technique to deal with functionals
depending on pairs $BV$ function/measure.\\

The following result is the analogue of Theorem \ref{BZZ} without
bending moment. We shall not give a proof of it since it can be
deduced from the one of Theorem \ref{BZZ} with much easier
arguments.

\begin{Theorem}\label{BZZ1}
Let $\o \subset \Rb^2$ be a bounded open set and $W:\Rb^{3 \times 3}
\to [0,+\infty)$ be a Borel function satisfying $(H_1)$ and $(H_2)$.
Then, for every $u \in BV(\O;\Rb^3)$, the sequence of functionals
$$J_\e(u):=\left\{
\begin{array}{ll}
\ds \int_\O W\left(\nabla_\a u \Big|\frac{1}{\e} \nabla_3 u\right)dx
& \text{ if } u \in W^{1,1}(\O;\Rb^3),\\[0.4cm]
+\infty & \text{ otherwise,}
\end{array}
\right.$$ $\G$-converges for the weak* topology of $BV(\O;\Rb^3)$ to
$$E(u):=\left\{
\begin{array}{ll}
\begin{array}{l}\ds \int_\o \mathcal QW_0(\nabla_\a
u)\, dx_\a+ \int_{J_u} (\mathcal Q W_0)^\infty
\big((u^+-u^-)\otimes\nu_u\big)\, d\HH^1\\[0.4cm]
\ds \hspace{1cm} + \int_\o (\mathcal Q
W_0)^\infty\left(\frac{dD^c_\a u}{d|D_\a^c u|}\right) d|D_\a^c u|
\end{array} & \text{
if }u\in BV(\o;\Rb^3),\\[0.5cm]
+\infty & \text{ otherwise},
\end{array}
\right. $$ where $W_0(\overline\xi):= \inf\{W(\overline \xi|b):\, b
\in \Rb^3)$ for all $\overline \xi \in \Rb^{3 \times 2}$, $\mathcal
QW_0$ is the 2D-quasiconvexification of $W_0$, and $(\mathcal Q
W_0)^\infty$ is the recession function of $\mathcal QW_0$.
\end{Theorem}

The paper is organized as follows: In section 2, we start by
introducing some useful notations and basic notions. Then, in
section 3 we prove some properties of the different energy densities
involved in our analysis. In section 4, we state some properties of
the $\G$-limit and the last two sections are devoted to the proof of
our $\G$-convergence result (Theorem \ref{BZZ}). The lower bound is
established in section 5 and the upper bound is proved in the last
one.

\section{Notations and Preliminaries}

\noindent Let $\O$ be a generic open subset of $\Rb^N$, we denote by
$\M(\O)$ the space of all signed Radon measures in $\O$ with bounded
total variation. By the Riesz Representation Theorem, $\M(\O)$ can
be identified to the dual of the separable space $\C_0(\O)$ of
continuous functions on $\O$ vanishing on the boundary $\partial
\O$. The $N$-dimensional Lebesgue measure in $\Rb^N$ is designated
as $\LL^N$ while $\HH^{N-1}$ denotes the $(N-1)$-dimensional
Hausdorff measure. If $\mu \in \M(\O)$ and $\lambda \in \M(\O)$ is a
nonnegative Radon measure, we denote by $\frac{d\mu}{d\lambda}$ the
Radon-Nikod\'ym derivative of $\mu$ with respect to $\lambda$. By a
generalization of the Besicovich Differentiation Theorem (see
\cite[Proposition 2.2]{ADM}), it can be proved that there exists a
Borel set $E \subset \O$ such that $\lambda(E)=0$ and
$$\frac{d\mu}{d\lambda}(x)=\lim_{\rho \to 0^+} \frac{\mu(x+\rho \, C)}{\lambda(x+\rho \, C)}$$
for all $x \in {\rm Supp }\, \mu \setminus E$ and any open convex
set $C$ containing the origin.

\vskip5pt

We say that $u \in L^1(\O;\Rb^d)$ is a function of bounded
variation, and we write $u \in BV(\O;\Rb^d)$, if all its first
distributional derivatives $D_j u_i$ belong to $\M(\O)$ for $1\leq i
\leq d$ and $1 \leq j \leq N$. We refer to \cite{AFP} for a detailed
analysis of $BV$ functions. The matrix-valued measure whose entries
are $D_j u_i$ is denoted by $Du$ and $|Du|$ stands for its total
variation. By the Lebesgue Decomposition Theorem we can split $Du$
into the sum of two mutually singular measures $D^a u$ and $D^s u$
where $D^a u$ is the absolutely continuous part of $Du$ with respect
to the Lebesgue measure $\LL^N$, while $D^s u$ is the singular part
of $Du$ with respect to $\LL^N$. By $\nabla u$ we denote the
Radon-Nikod\'ym derivative of $D^au$ with respect to the Lebesgue
measure so that we can write
$$Du=\nabla u \LL^N + D^s u.$$
Let $J_u$ be the jump set of $u$ defined as the set of points $x \in
\O$ such that there exist $u^\pm(x) \in \Rb^d$ (with $u^+(x)\neq
u^-(x)$) and $\nu_u(x) \in \mathbb S^{N-1}$ satisfying
$$\lim_{\rho \to 0^+} \frac{1}{\rho^N} \int_{\left\{y \in Q_{\nu_u(x)}(x,\rho) : \, \pm (y-x) \cdot \nu_u(x)
>0\right\}} |u(y) - u^\pm(x)|\, dy=0,$$
where $Q_\nu(x,\rho)$ denotes any cube of $\Rb^N$ centered at $x \in
\Rb^N$, with edge length $\rho>0$, and such that two of its faces
are orthogonal to $\nu \in \mathbb S^{N-1}$. It is known that $J_u$
is a countably $\HH^{N-1}$-rectifiable Borel set. The measure $D^s
u$ can in turn be decomposed into the sum of a jump part and a
Cantor part defined by $D^j u:=D^s u \res\, J_u$ and $D^c u:= D^s u
\res\, (\O \setminus J_u)$. We now recall the decomposition of $Du$:
$$Du= \nabla u  \LL^N + (u^+ -u^-)\otimes \nu_u {\cal H}^{N-1}\res\,
J_u + D^c u.$$ By Alberti's Rank One Theorem (see \cite{A}), the
matrix defined by
$$A(x):=\frac{dD^cu}{d|D^cu|}(x) \in \Rb^{d \times N}$$
has rank one for $|D^cu|$-a.e. $x \in \O$.  If $\O$ has
Lipschitz boundary, we denote by $Tu$ the trace of $u \in
BV(\O;\Rb^d)$ (or $u \in W^{1,1}(\O;\Rb^d)$) on $\partial \O$.

 We now recall basic facts about tangent measures and tangent
space to measures referring again to \cite{AFP} for more details.
Let $Q:=(-1/2,1/2)^N$ be the unit cube of $\Rb^N$ and let
$Q(x,\rho):=x+\rho \, Q$. If $\mu \in \M(\O)$ is a non negative
Radon measure in $\O$ and $x\in \O$, we denote by ${\rm Tan}(\mu,x)$
the set of all non negative finite Radon measures $\nu \in \M(Q)$
such that
$$\frac{1}{\mu(Q(x,\rho_j))} \int_{\Rb^N}
\varphi\left(\frac{y-x}{\rho_j}\right)\, d\mu(y) \to \int_Q
\varphi(y)\, d\nu(y),$$ for any $\varphi \in \C_c(\Rb^N)$ and for
some sequence $\{\rho_j\} \searrow 0^+$. The set ${\rm Tan}(\mu,x)$
is not empty and for any $t \in (0,1)$, there exists $\nu \in {\rm
Tan}(\mu,x)$ such that $\nu(\overline{tQ}) \geq t^N$ for $\mu$-a.e.
$x \in \O$ (see \cite[Corollary 2.43]{AFP}).

When $\mu=\HH^{N-1} \res \, S$ for some countably
$\HH^{N-1}$-rectifiable set $S \subset \Rb^N$, we say that $S$
admits  an approximate tangent space at $x \in S$ if there
exists a $(N-1)$-dimensional linear subspace $\pi$ of $\Rb^N$ such
that
$$\frac{1}{\rho^{N-1}} \int_S
\varphi\left(\frac{y-x}{\rho}\right)\, d\HH^{N-1}(y) \to \int_\pi
\varphi(y)\, d\HH^{N-1}(y),$$ for any $\varphi \in \C_c(\Rb^N)$.
From \cite[Theorem 2.83]{AFP}, we know that $\HH^{N-1}$-a.e. $x \in
S$ admits an approximate tangent space. Moreover, the
Federer-Vol'pert Theorem (see \cite[Theorem 3.78]{AFP}) asserts that
if $u \in BV(\O;\Rb^d)$, then for $\HH^{N-1}$-a.e. $x \in J_u$, the
hyperplane $\nu_u(x)^\perp$ coincides with the approximate tangent
space of $J_u$ at $x$.

 \vskip5pt

In the sequel we will always deal with the cases $N=2$ or $3$. Let
$\o \subset \Rb^2$ be a bounded open set and $I:=(-1/2,1/2)$, we
define $\O:=\o \times I$. We denote by $Q':=(-1/2,1/2)^2$ the unit
cube in $\Rb^2$ and if $\nu \in \mathbb S^1$, $Q'_\nu$ is the unit
cube centered at the origin with its faces either parallel or
orthogonal to $\nu$. If $x \in \Rb^2$ and $\rho>0$, we set
$Q'(x,\rho)=x+\rho \, Q'$ and $Q'_\nu(x,\rho):=x+\rho \, Q'_\nu$.
The canonical basis of $\Rb^2$ is denoted by $(e_1,e_2)$.

Given a matrix $\xi \in \Rb^{3 \x 3}$, $\xi$ will be written as
$(\overline{\xi}|\xi_3) $, where $\overline{\xi}:=(\xi_1|\xi_2 ) \in
\Rb^{3 \x 2}$ and $\xi_i$ denotes the $i$-th column of $\xi$. If $x
\in \Rb^3$, then $x_\a:=(x_1,x_2)\in \Rb^2$ is the vector of the
first two components of $x$. The notation $\nabla_\a$ and $\nabla_3$
denote respectively (approximate) differentiation with respect to
the variables $x_\a$ and $x_3$.

\section{Properties of the energy densities}

\subsection{The bulk energy density}

\noindent As in \cite{BFM}, we define $\mathcal Q^* W : \Rb^{3
\times 2} \times \Rb^3 \to [0,+\infty)$ by
\begin{eqnarray}\label{QastW}
{\cal Q}^\ast W(\overline \xi|b) & := & \inf_{\lambda, \, \varphi}
\left\{\int_{Q' \times I} W(\overline \xi + \nabla_\a
\varphi|\lambda \nabla_3 \varphi)\,dx: \lambda >0,\,
\varphi \in W^{1,1}(Q' \times I;\Rb^3),\right.\nonumber\\
&&\left. \varphi(\cdot, x_3) \hbox{ is } Q'\hbox{-periodic for }
\LL^1 \hbox{-a.e. } x_3 \in I, \lambda \int_{Q' \times I}\nabla_3
\varphi\, dy =b\right\}.
\end{eqnarray}

We recall the main properties of ${\cal Q}^\ast W$ proved in
\cite[Proposition 1.1]{BFM}.
\begin{Proposition}\label{propQastW}
Let $W:\Rb^{3 \times 3} \to [0,+\infty)$ be a Borel function
satisfying $(H_1)$ and let ${\cal Q}^\ast W$ be defined by
(\ref{QastW}). The following properties hold:
\begin{itemize}
\item $\mathcal CW \leq {\cal Q}^\ast W \leq {\cal Q}W$, where $\mathcal CW$ and ${\cal Q}W$
denote, respectively, the convex and quasiconvex envelopes of $W$;
\item for all $\overline{\xi} \in \Rb^{3 \times 2}$ and $b \in
\Rb^3$,
\begin{equation}\label{QastWgrowth}
\b' (|\overline{\xi}| + |b|) \leq {\cal Q}^\ast W(\overline{\xi}|
b)\leq \beta(1 + |\overline{\xi}|+ |b|);
\end{equation}
\item there holds
\begin{equation}\label{1.8BFM}
{\cal Q}^\ast ({\cal Q}W)= {\cal Q}^\ast W;
\end{equation}
\item let $W_0 : \Rb^{3 \times 2} \to [0,+\infty)$ be given by
$W_0(\overline{\xi}):=\inf \left\{W(\overline{\xi}|b): b \in
\Rb^3\right\}$ and ${\cal Q}W_0$ denotes its 2D-quasiconvex
envelope. Then we have
$$\inf_{b \in \Rb^3}{\cal Q}^\ast W(\overline{\xi}|b)= {\cal
Q}W_0(\overline{\xi}).$$
\end{itemize}
\end{Proposition}

We now highlight a convexity property of the energy density
$\mathcal Q^*W$.

\begin{Proposition}\label{convex}
The function $\mathcal Q^*W$ is convex in the directions $(z \otimes
\nu, b)$, with $z$, $b \in \Rb^3$ and $\nu \in \mathbb S^1$.
\end{Proposition}

\begin{proof}Let $b_1$, $b_2
\in \Rb^3$ and $\overline \xi_1,\;\overline \xi_2 \in \Rb^{3
\times 2}$ be such that $\overline \xi_2 - \overline \xi_1=z\otimes
\nu$ for some $z \in \Rb^3$ and $\nu \in \mathbb S^1$. Fix also
$\theta \in [0,1]$ and set
$$u(x_\a):=\left\{
\begin{array}{lll}
\overline \xi_1 x_\a + (x_\a \cdot \nu) z - (1-\theta)jz & \text{ if
} & j \in
\Zb \text{ and } j \leq x_\a \cdot \nu < j+\theta,\\
&&\\
\overline \xi_1 x_\a +(1+j)\theta z & \text{ if } & j \in \Zb \text{
and } j + \theta \leq x_\a \cdot \nu <j +1
\end{array}\right.$$
and $$A :=\{ x_\a \in \Rb^2: \text{ there exists }j \in \Zb \text{
such that }j \leq x_\a \cdot \nu < j+\theta\}.$$ Now define
$u_n(x_\a):= u(nx_\a) /n$ and $\overline b_n(x_\a):=\chi_A(nx_\a)\,
b_2 + (1-\chi_A(nx_\a))\, b_1$. Then, by the Riemann-Lebesgue Lemma,
$u_n \rightharpoonup (\theta \overline \xi_2 + (1-\theta)\overline
\xi_1)\,  x_\a$ in $W^{1,p}(Q';\Rb^3)$ and $\overline b_n
\rightharpoonup  \theta \, b_2 + (1-\theta)\, b_1$ in
$L^p(Q';\Rb^3)$ for every $p\geq 1$. Using the fact that the
functional
$$ (u,\overline b)  \mapsto
\int_{Q'} \mathcal Q^*W(\nabla_\a u |\overline  b)\, dx_\a$$ is
sequentially weakly lower semicontinuous in $W^{1,p}(Q';\Rb^3)
\times L^p(Q' ;\Rb^3)$ (see {\it e.g.} \cite[Remark 1.4]{BFM}), we
infer that
\begin{eqnarray*}
\mathcal Q^* W\big(\theta (\overline \xi_2|b_2) +
(1-\theta)(\overline \xi_1|b_1)\big) & \leq &
\liminf_{n\to +\infty} \int_{Q'}\mathcal Q^*W(\nabla u_n|\overline b_n)\, dx_\a\\
& = & \lim_{n \to +\infty} \int_{Q'} \big[\chi_A(nx_\a) \mathcal
Q^*W(\overline \xi_1 + z \otimes \nu |b_2)\\
&&\hspace{2cm} +(1-\chi_A(nx_\a)) \mathcal
Q^*W(\overline \xi_1|b_1) \big]\, dx_\a\\
& = & \theta \mathcal Q^*W(\overline \xi_2|b_2) + (1-\theta)\mathcal
Q^*W(\overline \xi_1|b_1),
\end{eqnarray*}
which is the desired result.
\end{proof}

We also remark that we could arrive at the same conclusion by
observing that the function $\mathcal Q^*W$ is ${\cal
A}$-quasiconvex (see \cite{FM2}, page 1369, Example (iii)) with
respect to the operator $\mathcal A:= ({\rm curl},0)$, where
$${\cal A}:( F|\ovb) \mapsto (\textrm{curl} F,
0)$$ with $F:\Rb^2 \to \Rb^{3 \times 2}$ and $b:\Rb^2 \to \Rb^3$.
Indeed, by virtue of \cite[Proposition 3.4]{FM2}, the function
$\mathcal Q^*W$ turns out to be convex in the directions $(z \otimes
\nu, b)$, with $z$, $b
\in \Rb^3$ and $\nu \in \mathbb S^1$.\\

The following result asserts that in the definition (\ref{QastW}) of
$\mathcal Q^* W$, one can replace the cube $Q'$ by any rotated cube
$Q'_\nu$.

\begin{Proposition}\label{formequiA*W}
Let $W:\Rb^{3 \times 3} \to [0,+\infty)$ be a Borel function
satisfying $(H_1)$, and assume that there exists a constant $L>0$
such that \begin{equation}\label{Wlip1}|W(\xi)-W(\xi')|\leq
L|\xi-\xi'|\quad \text{ for every $\xi$, $\xi' \in \Rb^{3 \times
3}$.}\end{equation} Then for every $\nu \in \mathbb S^1$, $\overline
\xi \in \Rb^{3 \times 2}$ and $b \in \Rb^3$,
\begin{eqnarray*}
{\cal Q}^\ast W(\overline \xi|b) & = & \inf_{\lambda, \, \varphi}
\left\{\int_{Q'_\nu \times I} W(\overline \xi + \nabla_\a
\varphi|\lambda \nabla_3 \varphi)\,dx: \lambda >0,\,
\varphi \in W^{1,1}(Q'_\nu \times I;\Rb^3),\right.\\
&&\left. \varphi(\cdot, x_3) \hbox{ is } Q'_\nu\hbox{-periodic for }
\LL^1 \hbox{-a.e. } x_3 \in I, \lambda \int_{Q'_\nu \times
I}\nabla_3 \varphi\, dy =b\right\}.
\end{eqnarray*}
\end{Proposition}

\begin{proof}
Fix $\overline \xi \in \Rb^{3 \times 2}$ and $b \in \Rb^3$, and
define for every $\nu \in \mathbb S^1$,
\begin{eqnarray*}
I(\nu) & := & \inf_{\lambda, \, \varphi} \left\{\int_{Q'_\nu \times
I} W(\overline \xi + \nabla_\a \varphi|\lambda \nabla_3
\varphi)\,dx: \lambda >0,\,
\varphi \in W^{1,1}(Q'_\nu \times I;\Rb^3),\right.\\
&&\left. \varphi(\cdot, x_3) \hbox{ is } Q'_\nu\hbox{-periodic for }
\LL^1 \hbox{-a.e. } x_3 \in I, \lambda \int_{Q'_\nu \times
I}\nabla_3 \varphi\, dy =b\right\}.
\end{eqnarray*}
We shall prove that for any $\nu$ and $\nu' \in \mathbb S^1$, then
$I(\nu) \leq I(\nu')$. Interchanging the roles of $\nu$ and $\nu'$,
we will deduce that the inequality is actually an equality, and
taking $\nu'=e_2$ that $\mathcal Q^* W(\overline \xi|b)=I(\nu)$
which is the conclusion of the proposition.

Let $\lambda >0$ and $\varphi \in W^{1,1}(Q'_{\nu'} \times I;\Rb^3)$
be such that $\varphi(\cdot, x_3)$ is $Q'_{\nu'}$-periodic for
$\LL^1$-a.e. $x_3 \in I$ and $\lambda \int_{Q'_{\nu'} \times
I}\nabla_3 \varphi\, dy =b$. Extend $\varphi$ by
$Q'_{\nu'}$-periodicity to the whole $\Rb^2 \times I$ and set
$\varphi_n(x_\a,x_3):=\varphi(nx_\a,x_3)/n$. Consider also a cut-off
function $\zeta_k \in \C^\infty_c(Q'_{\nu};[0,1])$ satisfying
\begin{equation}\label{cutoffzeta}\left\{
\begin{array}{l}
\ds \zeta_k=1 \text{ on } Q'_\nu\left(0,1-\frac{1}{k}\right),\\[0.4cm]
\ds \|\nabla_\a \zeta_k\|_{L^\infty(Q'_\nu;\Rb^2)}\leq 2k^2.
\end{array}
\right.\end{equation} Define now
$$\psi_{n,k}(x_\a,x_3):=\varphi_n(x_\a,x_3)\zeta_k(x_\a) +
\frac{x_3}{\lambda n} \left[b - \lambda n \int_{Q'_\nu \times
I}\zeta_k(z_\a)\nabla_3\varphi_n(z_\a,z_3)\, dz \right].$$ It turns
out that $\psi_{n,k} \in W^{1,1}(Q'_\nu \times I;\Rb^3)$, that
$\psi_{n,k}(\cdot,x_3)$ is $Q'_\nu$-periodic for $\LL^1$-a.e. $x_3
\in I$ and that $\lambda n \int_{Q'_\nu \times
I}\nabla_3\psi_{n,k}\, dy =b$. Hence the pair $(\lambda
n,\psi_{n,k})$ is admissible for $I(\nu)$ and thus
$$I(\nu) \leq \int_{Q'_\nu \times I}W(\overline \xi
+\nabla_\a\psi_{n,k}|\lambda n \nabla_3 \psi_{n,k})\, dx.$$
Consequently, (\ref{cutoffzeta}) yields to
\begin{eqnarray*}
I(\nu) & \leq & \int_{Q'_\nu \big(0,1-\frac{1}{k}\big) \times
I}W\left(\overline \xi +\nabla_\a\varphi_n \Big|\lambda n \nabla_3
\varphi_n + b - \lambda n
\int_{Q'_{\nu} \times I}\zeta_k(z_\a) \nabla_3\varphi_n(z_\a,z_3)\, dz \right) dx\\
&& + \int_{\Big(Q'_\nu \setminus
Q'_\nu\big(0,1-\frac{1}{k}\big)\Big) \times I}W(\overline \xi
+\nabla_\a\psi_{n,k}|\lambda n \nabla_3 \psi_{n,k})\, dx
%&& + \int_{\Big(Q'_\nu \setminus
%Q'_\nu\big(0,1-\frac{1}{k}\big)\Big) \times I}W\left(\overline \xi
%\Big|b - \lambda n \int_{Q'_{\nu} \times
%I}\zeta_k(z_\a)\nabla_3\varphi_n(z_\a,z_3)\, dz\right)  dx
\end{eqnarray*}
and using the growth condition $(H_1)$ together with the Lipschtiz
property (\ref{Wlip1}) of $W$, we get that
\begin{eqnarray*}
I(\nu) & \leq & \int_{Q'_\nu \times I}W\big(\overline \xi
+\nabla_\a\varphi(nx_\a,x_3) |\lambda \nabla_3
\varphi(nx_\a,x_3)\big)\, dx\\
&& + \b \int_{\Big(Q'_\nu \setminus
Q'_\nu\big(0,1-\frac{1}{k}\big)\Big) \times I}\big( 1+|\overline
\xi| + |\nabla_\a \varphi(nx_\a,x_3)|+\lambda|\nabla_3
\varphi(nx_\a,x_3)| + 2k^2 |\varphi_n(x)|\big)\, dx\\
 && +L\left|\; b - \lambda \int_{Q'_{\nu} \times I}\zeta_k(z_\a)
\nabla_3\varphi(n z_\a,z_3)\, dz \right|.
\end{eqnarray*}
Applying the Riemann-Lebesgue Lemma and the fact that $\varphi_n \to
0$ in $L^1(Q'_\nu \times I;\Rb^3)$, it implies, sending $n \to
+\infty$, that
\begin{eqnarray*}
I(\nu) & \leq & \int_{Q'_{\nu'} \times I}W\big(\overline \xi
+\nabla_\a\varphi(y) |\lambda \nabla_3
\varphi(y)\big)\, dy\\
&& + \b \left[1 -\left(1-\frac{1}{k}\right)^2  \right]
\int_{Q'_{\nu'} \times I}\big(1+|\overline \xi|
+|\nabla_\a\varphi(y)|+\lambda |\nabla_3 \varphi(y)|\big)\, dy\\
&& +L \left|\; b - \left(\lambda \int_{Q'_{\nu'} \times I}
\nabla_3\varphi(z)\, dz\right)\int_{Q'_\nu}\zeta_k(y_\a)\, dy_\a
\right|.
\end{eqnarray*}
As $\lambda \int_{Q'_{\nu'} \times I} \nabla_3\varphi(z)\, dz=b$ and
$\zeta_k \to 1$ in $L^1(Q'_\nu)$, we obtain letting $k \to +\infty$
that
$$I(\nu) \leq \int_{Q'_{\nu'} \times I}W\big(\overline \xi
+\nabla_\a\varphi(y) |\lambda \nabla_3 \varphi(y)\big)\, dy.$$
Taking the infimum over all pairs $(\lambda,\varphi)$ as above
implies that $I(\nu) \leq I(\nu')$ which is the desired result.
\end{proof}

\subsection{The surface energy density}

\noindent Let $W^\infty$ (resp. $(\mathcal Q^*W)^\infty$) be the
recession function of $W$ (resp. $\mathcal Q^* W$) defined by
$$W^\infty(\xi):= \limsup_{t \to +\infty}\frac{W(t\xi)}{t} \quad \left(\text{resp. }
(\mathcal Q^*W)^\infty(\xi):= \limsup_{t \to +\infty}\frac{\mathcal
Q^* W(t\xi)}{t}\right)$$ for every $\xi \in \Rb^{3 \times 3}$.

Let $(z,b,\nu) \in \Rb^3 \times \Rb^3 \times \mathbb S^1$ and
consider $\tau \in \mathbb S^1$ such that $(\tau,\nu)$ is an
orthonormal basis of $\Rb^2$. Define the auxiliary surface energy
$\g : \Rb^3 \times \Rb^3 \times\mathbb S^1 \to [0,+\infty)$ by
\begin{eqnarray}\label{surfnrj}
\g(z,\nu,b)&:=& \inf_{\lambda,\varphi} \Bigg\{\int_{Q'_\nu \times I}
W^\infty(\nabla_\a \varphi|\lambda \nabla_3 \varphi)\, dx : \lambda
>0, \, \varphi \in W^{1,1}(Q'_\nu \times I;\Rb^3),\,\varphi^{+\nu} - \varphi^{-\nu} = z,\nonumber\\
&&\varphi \text{ is }1\text{-periodic in the direction }\tau \text{
and } \lambda\int_{Q'_\nu \times I}\nabla_3\varphi\, dy=b\Bigg\},
\end{eqnarray}
where $\varphi^{\pm\nu}$ stands for the trace of $\varphi$ on the
face $\{(x_\a,x_3) \in Q'_\nu: x_\a \cdot \nu =\pm 1/2\}$. This
density will naturally appear in the proof of the lower bound of the
jump part. However, arguing as in \cite{AFP} page 313, one can
observe that $\g$ actually coincides with $(\mathcal Q^\ast
W)^\infty$.

\begin{Proposition}\label{denssurf}
Let $W:\Rb^{3 \times 3} \to [0,+\infty)$ be a Borel function
satisfying $(H_1)$, $(H_2)$ and (\ref{Wlip1}). Then for every $z$,
$b \in \Rb^3$ and $\nu \in \mathbb S^1$, we have
$$\g(z,\nu,b) = (\mathcal Q^* W)^\infty(z \otimes \nu|b)=\mathcal Q^*(W^\infty)(z\otimes \nu|b).$$
\end{Proposition}

\begin{proof}
The proof is divided into two steps. Firstly we shall prove that
$\g(z,\nu,b)=\mathcal Q^*(W^\infty)(z \otimes \nu|b)$ and then that
$\mathcal Q^*(W^\infty)(z \otimes \nu|b)=(\mathcal Q^* W)^\infty(z
\otimes \nu|b)$.

{\bf Step 1. } Let $\lambda>0$ and $\psi \in W^{1,1}(Q'_\nu \times
I;\Rb^3)$ be such that $\psi(\cdot,x_3)$ is $Q'_\nu$-periodic for
$\LL^1$-a.e. $x_3 \in I$ and $\lambda\int_{Q'_\nu \times
I}\nabla_3\psi\, dy=b$. Define
$$\varphi(x_\a,x_3):=(x_\a \cdot \nu)z + \psi(x_\a,x_3)\quad \text{ for every }(x_\a,x_3) \in Q'_\nu \times I.$$
Clearly $\varphi \in W^{1,1}(Q'_\nu \times I;\Rb^3)$, $\varphi$ is
$1$-periodic in the direction $\tau$ and $\varphi^{+\nu} -
\varphi^{-\nu} =z$. Moreover, we have that $\lambda\int_{Q'_\nu
\times I}\nabla_3\varphi\, dy=\lambda\int_{Q'_\nu \times
I}\nabla_3\psi\, dy=b$. Thus, by (\ref{surfnrj}), $\varphi$ is
admissible for $\g(z,\nu,b)$ and consequently
$$\g(z,\nu,b) \leq \int_{Q'_\nu \times I} W^\infty (\nabla_\a \varphi|\lambda\nabla_3 \varphi)\, dx=
\int_{Q'_\nu \times I} W^\infty (z\otimes \nu + \nabla_\a
\psi|\lambda\nabla_3 \psi)\, dx.$$ Taking the infimum over all such
$(\lambda,\psi)$, and using Proposition \ref{formequiA*W} yields
$\g(z,\nu,b) \leq \mathcal Q^*(W^\infty)(z \otimes \nu|b)$.

Conversely, consider $\lambda>0$ and $\varphi \in W^{1,1}(Q'_\nu
\times I;\Rb^3)$ such that $\varphi$ is $1$-periodic in the
direction $\tau$, $\varphi^{+\nu} - \varphi^{-\nu} = z$ and
$\lambda\int_{Q'_\nu \times I}\nabla_3\varphi\, dy=b$. Define
$$\psi(x_\a,x_3):=- (x_\a \cdot \nu)z + \varphi(x_\a,x_3 \quad \text{ for every }(x_\a,x_3) \in Q'_\nu \times I.$$
Then $\psi \in W^{1,1}(Q'_\nu \times I;\Rb^3)$, $\psi$ is
$1$-periodic in the direction $\tau$. Moreover noticing that
$\psi^{+\nu}- \psi^{-\nu}=0$, it implies that $\psi$ is actually
$Q'_\nu$-periodic. As $\lambda\int_{Q'_\nu \times I}\nabla_3\psi\,
dy=\lambda\int_{Q'_\nu \times I}\nabla_3\varphi\, dy=b$ it follows
that $\psi$ is admissible for $\mathcal Q^*(W^\infty)(z \otimes
\nu|b)$ and consequently
$$\mathcal Q^*(W^\infty)(z \otimes \nu|b) \leq \int_{Q'_\nu \times I} W^\infty (z \otimes \nu + \nabla_\a
\psi|\lambda\nabla_3 \psi)\, dx= \int_{Q'_\nu \times I} W^\infty
(\nabla_\a \varphi|\lambda\nabla_3 \varphi)\, dx.$$ By the
arbitrariness of $(\lambda,\psi)$, it yields $\mathcal
Q^*(W^\infty)(z \otimes \nu|b) \leq \g(z,\nu,b)$ and it completes the proof of the first step.\\

{\bf Step 2. }Now take any pair $(\lambda,\varphi)$ where
$\lambda>0$ and $\varphi \in W^{1,1}(Q' \times I;\Rb^3)$ is such
that $\varphi(\cdot,x_3)$ is $Q'$-periodic and $\lambda\int_{Q'
\times I}\nabla_3\varphi\, dy=b$. Then
$$\frac{\mathcal Q^*W(t(z \otimes \nu|b))}{t} \leq \int_{Q' \times I}
\frac{W(t z \otimes \nu + \nabla_\a (t\varphi)|\lambda
\nabla_3(t\varphi))}{t}\, dx,$$ and by the growth condition $(H_1)$,
we have for $t>1$,
$$\frac{W(t z \otimes \nu + t \nabla_\a \varphi|\lambda t
\nabla_3\varphi)}{t} \leq \b(1+|z| + |\nabla_\a\varphi| + \lambda
|\nabla_3 \varphi|) \in L^1(Q' \times I).$$ Hence by the limsup
version of Fatou's Lemma, it follows that
\begin{eqnarray*}
(\mathcal Q^*W)^\infty(z\otimes \nu|b) & = & \limsup_{t \to
+\infty}\frac{\mathcal Q^*W(t(z \otimes \nu|b))}{t} \leq \limsup_{t
\to +\infty}\int_{Q' \times I} \frac{W(t z \otimes
\nu + t \nabla_\a \varphi|\lambda t \nabla_3\varphi)}{t}\, dx\\
& \leq &\int_{Q' \times I}\limsup_{t \to +\infty} \frac{W(t z
\otimes \nu + t \nabla_\a \varphi|\lambda t \nabla_3\varphi)}{t}\,
dx = \int_{Q' \times I}W^\infty(z \otimes \nu + \nabla_\a
\varphi|\lambda \nabla_3\varphi)\, dx.
\end{eqnarray*}
Finally taking the infimum over all $(\lambda,\varphi)$ as before,
we obtain that $(\mathcal Q^*W)^\infty(z\otimes \nu|b) \leq \mathcal
Q^*(W^\infty)(z\otimes \nu|b)$.

To prove the converse inequality, for any $t >1$, let $\lambda_t>0$
and $\varphi_t \in W^{1,1}(Q' \times I;\Rb^3)$ be such that
$\varphi_t(\cdot,x_3)$ is $Q'$-periodic for $\LL^1$-a.e. $x_3 \in
I$, $\lambda_t \int_I \nabla_3 \varphi_t\, dy=b$ and
\begin{equation}\label{1910}\int_{Q' \times I}W(t z \otimes \nu + t \nabla_\a
\varphi_t|t\lambda_t \nabla_3 \varphi_t)\,dx \leq \mathcal Q^* W(t(z
\otimes \nu|b) ) +1.\end{equation} By the growth and coercivity
properties $(H_1)$ and (\ref{QastWgrowth}), it turns out that
\begin{equation}\label{1909}\| (\nabla_\a \varphi_t |\lambda_t \nabla_3
\varphi_t)\|_{L^1(Q' \times I;\Rb^{3 \times 3})} \leq
C,\end{equation} for some constant $C>0$ independent of $t$. Hence
using $(H_2)$ and the fact that $W^\infty$ is positively
$1$-homogeneous, it follows that
\begin{eqnarray*}
\mathcal Q^*(W^\infty)(z\otimes \nu|b) & \leq & \int_{Q' \times
I}W^\infty(z \otimes \nu + \nabla_\a \varphi_t|\lambda_t \nabla_3
\varphi_t)\,dx\\
& \leq & \int_{Q' \times I}\frac{W(t z \otimes \nu + t \nabla_\a
\varphi_t|t\lambda_t \nabla_3
\varphi_t)}{t}\,dx\\
&&+\frac{C}{t}\int_{Q' \times I} (1 + t^{1-r}|z|^{1-r} +
t^{1-r}|(\nabla_\a \varphi_t |\lambda_t \nabla_3
\varphi_t)|^{1-r})\,dx.\end{eqnarray*} From H\"older's Inequality
together with (\ref{1910}) and (\ref{1909}), it yields $$\mathcal
Q^*(W^\infty)(z\otimes \nu|b) \leq \frac{\mathcal Q^* W(t(z\otimes
\nu|b))}{t} + \frac{C}{t}+\frac{C}{t^r}.$$ Finally, taking the
limsup as $t\to +\infty$ leads to $\mathcal Q^*(W^\infty)(z\otimes
\nu|b) \leq (\mathcal Q^* W)^\infty(z\otimes \nu|b)$ which concludes
the proof of the second step and of the proposition.
\end{proof}

\section{Properties of the $\G$-limit}

\noindent We start by localizing the functionals on $\A_0$, the
family of all bounded open subsets of $\Rb^2$. Let $J_\e:
BV(\Rb^3;\Rb^3)\times \M(\Rb^2;\Rb^3)\times \A_0 \to [0,+\infty]$ be
defined by
\begin{equation}\label{energies}
J_\e(u,\ovb,A):=\left\{
\begin{array}{ll}
\ds \int_{A \times I} W\left(\nabla_\a u \Big|\frac{1}{\e} \nabla_3
u\right)dx & \text{ if } \left\{
\begin{array}{l}
u \in W^{1,1}(A \times I;\Rb^3),\\
\ovb=\frac{1}{\e}\int_I \nabla_3 u(\cdot,x_3)\, dx_3,
\end{array}
\right.\\[0.5cm] +\infty & \text{ otherwise}.
\end{array}
\right.
\end{equation}
In the sequel, we will also use the family $\A(\o)$ of all open
subsets of $\o$. For every sequence $\{\e_j\} \searrow 0^+$ define
the $\G$-lower limit of $J_{\e_j}$ given by
$$J_{\{\e_j\}}(u,\ovb,A):=\inf_{\{u_j,\ovb_j\}}\left\{\liminf_{j \to +\infty}
J_{\e_j}(u_j,\ovb_j,A): u_j \xrightharpoonup[]{*}u \text{ in } BV(A
\times I;\Rb^3), \, \ovb_j \xrightharpoonup[]{*}\ovb \text{ in }
\M(A;\Rb^3) \right\}.$$ In order to show that the family $\{J_\e\}$
$\G$-converges to the functional $E$, it is enough to prove that for
every sequence $\{\e_j\} \searrow 0^+$, there exists a further
subsequence $\{\e_{j_n}\}$ such that
$J_{\{\e_{j_n}\}}(u,\ovb,\o)=E(u,\ovb)$ for any $(u,\ovb) \in
BV(\o;\Rb^3) \times \M(\o;\Rb^3)$.

It is easily seen from the coercivity condition $(H_1)$ that if
$J_{\{\e_j\}}(u,\ovb,\o)<+\infty$, then necessarily $D_3 u =0$ so
that $u$ (may be identified to a function which) belongs to
$BV(\o;\Rb^3)$. Thus it suffices to consider $(u,\ovb)\in
BV(\o;\Rb^3) \times \M(\o;\Rb^3)$ in which case we have that
\begin{eqnarray}\label{gammalower_limit}
J_{\{\e_j\}}(u,\ovb, A) & = & \inf_{\{u_j\}} \left\{\liminf_{j \to
+\infty} \int_{A \times I}
W\left(\nabla_\a u_j\Big|\frac{1}{\e_j} \nabla_3 u_j\right) dx: \{u_j\} \subset W^{1,1}(A \times I;\Rb^3)\right.\nonumber\\
&& \left. u_j\xrightharpoonup[]{*} u \hbox{ in } BV(A \times I;
\Rb^3),\, \frac{1}{\e_j} \int_I \nabla_3 u_j(\cdot,x_3)\, dx_3
\xrightharpoonup[]{*} \ovb \hbox{ in } \M(A;\Rb^3)\right\}.
\end{eqnarray}
Note that thanks to the coercivity condition $(H_1)$, the weak*
convergence in $BV(A \times I;\Rb^3)$ in (\ref{gammalower_limit}) is
equivalent to the strong convergence in $L^1(A \times I;\Rb^3)$.

It is expected, as in most variational problems in $BV$ (see
\cite{FM}), that the $\G$-limit should be the sum of three terms
relative to the decomposition of the gradient $D_\a u$ of a function
$u\in BV(\o;\Rb^3)$ into bulk, jump and Cantor parts. In the present
study, there will be a fourth one which comes from the presence of
the bending moment $\ovb$, and which is due to the fact that
$\overline b$ may be singular with respect to $D_\a u$. There is no
hope to avoid this so called singular term as the following example
shows.

\begin{Example}\label{ex}
{\rm There exist $(u,\ovb) \in BV(\o;\Rb^3) \times \M(\o;\Rb^3)$ and
a sequence $\{u_\e\} \subset W^{1,1}(\O;\Rb^3)$ such that $u_\e
\xrightharpoonup[]{*} u$ in $BV(\O;\Rb^3)$,
$\frac{1}{\e}\int_I\nabla_3 u_\e(\cdot,x_3)\, dx_3
\xrightharpoonup[]{*} \ovb$ in $\M(\o;\Rb^3)$ where the measures
$D_\a u$ and $\ovb$ are mutually singular.

For simplicity, we construct an example for scalar valued functions.
Consider a nonnegative radial function $\varrho \in
\C^\infty_c(\Rb^3)$ such that ${\rm Supp}(\varrho) \subset B(0,1/2)$
and $\int_{\Rb^3} \varrho(x)\, dx=1$, and set
$\varphi(x_\a,x_3):=\int_{-1/2}^{x_3}\varrho(x_\a,s)\, ds$. Assume
that $\o$ contains the origin and define $u_\e \in W^{1,1}(\O)$ by
$$u_\e(x_\a,x_3):=u(x_\a)+\frac{1}{\e}
\varphi\left(\frac{x}{\e}\right),$$ where $u \in W^{1,1}(\o)$. Then,
by a change of variables, we have
$$\|u_\e-u\|_{L^1(\O)} \leq \e, \quad \|\nabla u_\e\|_{L^1(\O;\Rb^3)} \leq
\|\nabla u\|_{L^1(\o;\Rb^3)} + \e \|\nabla
\varphi\|_{L^1(\O;\Rb^3)}$$ so that $u_\e \rightharpoonup u$ in
$W^{1,1}(\O)$ (and thus also weakly* in $BV(\O)$). On the other
hand, we have that
$$\frac{1}{\e} \nabla_3
u_\e(x)=\frac{1}{\e^3}\varrho\left(\frac{x}{\e}\right)$$ and
consequently, $\frac{1}{\e}\int_I \nabla_3 u_\e(\cdot,x_3)\, dx_3
\xrightharpoonup[]{*} \delta$ in $\M(\o)$, where $\delta$ is the
Dirac mass at $0 \in \Rb^2$, which is singular with respect to $D_\a
u=\nabla_\a u \LL^2$. }
\end{Example}

\begin{Remark}\label{sobolevcase}
{\rm In \cite[Theorem 1.2]{BFM}, it has been shown that
$$J_{\{\e_j\}}(u,\ovb,\o)=\int_\o \mathcal Q^* W(\nabla_\a u|\ovb)\, dx_\a=E(u,\ovb),$$
for $u \in W^{1,p}(\o;\Rb^3)$ and $\ovb \in L^p(\o;\Rb^3)$ with
$p>1$. An analogous argument ensures that the same representation
result holds when  $u \in W^{1,1}(\o;\Rb^3)$ and $\ovb \in
L^1(\o;\Rb^3)$.}
\end{Remark}

\begin{Remark}
{\rm Arguing exactly as in \cite[Lemma 2.3]{BFM}, one can show that
$J_{\{\e_j\}}$ remains unchanged if we replace $W$ by its
quasiconvexification $\mathcal Q W$ in (\ref{energies}). Hence using
(\ref{1.8BFM}), upon replacing $W$ by $\mathcal QW$, we may assume
without loss of generality that $W$ is quasiconvex. Then, by the
growth condition $(H_1)$ and {\it e.g.} the proof of Theorem 2.1,
Step 2 in \cite{M}, there exists a constant $L>0$ such that
\begin{equation}\label{Wlip}
|W(\xi)-W(\xi')|\leq L|\xi-\xi'|,
\end{equation}
for every $\xi$ and $\xi' \in \Rb^{3 \times 3}$. As a consequence,
$W^\infty$ is Lipschitz continuous as well and
\begin{equation}\label{Winftylip}
|W^\infty(\xi)-W^\infty(\xi')|\leq L |\xi-\xi'|.
\end{equation}
}
\end{Remark}

Let $\R_0$ be the countable subfamily of $\A_0$ obtained by taking
all finite unions of open squares in $\Rb^2$ with faces parallel to
the axes, centered at $x_\a \in \mathbb Q^2$, and with rational edge
length. Since $\M(\o; \Rb^3)$ and $BV(\O;\Rb^3)$ are the duals of
separable spaces (see {\it e.g.} \cite[Remark 3.12]{AFP}), the
adopted weak* topologies in (\ref{gammalower_limit}), and their
metrizability on bounded sets, ensure the applicability of
Kuratowsky's Compactness Theorem (we refer to {\it e.g.}
\cite[Corollary 8.12]{DM} for the weak topology of a Banach space
with separable dual; it can be checked that a similar result holds
for the weak* topology of a Banach space which is the dual of a
separable one). Thus, through a diagonal argument, it guarantees the
existence of a subsequence $\{\e_n\} \equiv \{\e_{j_n}\}$ of
$\{\e_j\}$ such that $J_{\{\e_n\}}(u,\ovb,A)$ is the $\G$-limit of
$J_{\e_n}(u,\ovb,A)$ for all $A \in \R_0$ (and also $A=\o$) and all
$(u,\ovb)$ in $BV(A;\Rb^3)\times \M(A;\Rb^3)$.

\begin{Lemma}\label{1557}
Let $\o \subset \Rb^2$ be a bounded open set and let $A
\subset\subset \o$ be an open subset of $\o$ with Lipschitz
boundary. For every  $(u,\ovb) \in BV(\o;\Rb^3) \times \M(\o;\Rb^3)$
satisfying $|\ovb|(\partial A)=0$, there exists a sequence $\{v_n\}
\subset W^{1,1}(A \times I;\Rb^3)$ such that $$\left\{
\begin{array}{l}
v_n \to u \text{ in } L^1(A \times I;\Rb^3),\\[0.2cm]
\frac{1}{\e_n} \int_I \nabla_3 v_n(\cdot,x_3)\, dx_3
\xrightharpoonup[]{*}
\ovb \text{ in }\M(A;\Rb^3),\\[0.2cm]
Tv_n=Tu \text{ on }\partial A \times I,\\[0.2cm]
|D_\a v_n|(A \times I) \to |D_\a u|(A),\\[0.2cm]
\frac{1}{\e_n} |D_3 v_n|(A \times I) \to |\ovb|(A).
\end{array}\right.$$
\end{Lemma}

\begin{proof}By \cite[Lemma
2.5]{BouFonMas}, there exists a sequence $\{\tilde v_n\} \subset
W^{1,1}(A;\Rb^3)$ such that $\tilde v_n \to u$ in $L^1(A;\Rb^3)$,
$|D_\a \tilde v_n|(A) \to |D_\a u|(A)$ and $T\tilde v_n = Tu$ on
$\partial A$. Consider a usual sequence of mollifiers denoted by
$\{\varrho_k\}$. Then from \cite[Theorem 2.2]{AFP}, we have that
$\ovb
* \varrho_k \xrightharpoonup[]{*} \ovb$ in $\M_{\rm
loc}(\o;\Rb^3)$ and thus
\begin{equation}\label{1141}\ovb
* \varrho_k \xrightharpoonup[]{*} \ovb \text{ in } \M(A;\Rb^3).
\end{equation}
Moreover, since $|\ovb|(\partial A)=0$, it follows that $|\ovb
* \varrho_k|(A) \to |\ovb|(A)$. As $\ovb * \varrho_k \in L^1(A;\Rb^3)$ one can find $\ovb_k \in
\C^\infty_c(A;\Rb^3)$ such that
\begin{equation}\label{1142}\|\ovb_k - (\ovb * \varrho_k)\|_{L^1(A;\Rb^3)} \leq
\frac{1}{k}.\end{equation} Now define
$$v^k_n(x_\a,x_3):= \tilde v_n(x_\a)+ \e_n x_3 \ovb_k(x_\a).$$ The sequence $\{v^k_n\} \subset W^{1,1}(A
\times I;\Rb^3)$, $v^k_n \to u$ in $L^1(A \times I;\Rb^3)$ as $n \to
+\infty$ and $Tv^k_n = Tu$ on $\partial A \times I$. Moreover from
the lower semicontinuity of the total variation, we infer that
$$\lim_{k \to +\infty} \lim_{n \to +\infty} |D_\a v_n^k|(A
\times I) = |D_\a u|(A)$$ and from (\ref{1141}) and (\ref{1142}),
$$\frac{1}{\e_n} \int_I \nabla_3 v_n^k(\cdot,x_3)\, dx_3 =\ovb_k
\xrightharpoonup[k \to +\infty]{*} \ovb \text{ in }\M(A;\Rb^3),$$
uniformly with respect to $n \in \Nb$. Using the separability of
$\C_0(A;\Rb^3)$ and a diagonalization argument (see {\it e.g.}
\cite[Lemma 7.1]{BFF}), one may find a sequence $k(n) \nearrow
+\infty$ such that, setting $v_n:=v_n^{k(n)}$, then $v_n \to u$ in
$L^1(A \times I;\Rb^3)$, $\frac{1}{\e_n} \int_I\nabla_3
v_n(\cdot,x_3)\, dx_3 \xrightharpoonup[]{*} \ovb$ in $\M(A;\Rb^3)$,
$Tv_n=Tu$ on $\partial A \times I$ for all $n \in \Nb$, $|D_\a
v_n|(A \times I) \to |D_\a u|(A)$ and $\frac{1}{\e_n} |D_3 v_n|(A
\times I)=|\ovb_{k(n)}|(A) \to |\ovb|(A)$.
\end{proof}

Using Lemma \ref{1557} and an adaptation of the proof of \cite[Lemma
2.2]{BFM}, we can prove the following result which will be
instrumental in the proof of the lower bound. It states that,
without loss of generality, recovery sequences can be taken in such
a way to match the lateral boundary of their target.

\begin{Lemma}\label{tracegamma}
Let $\o \subset \Rb^2$ be a bounded open set and let $A
\subset\subset \o$ be an open subset with Lipschitz boundary.
Consider  $(u,\ovb) \in BV(\o;\Rb^3) \times \M(\o;\Rb^3)$ such that
$|\ovb|(\partial A)=0$ and assume that $\{u_n\} \subset W^{1,1}(A
\times I;\Rb^3)$ is a sequence satisfying $u_n \to u$ in $L^1(A
\times I;\Rb^3)$, $\frac{1}{\e_n} \int_I \nabla_3 u_n(\cdot,x_3)\,
dx_3 \xrightharpoonup[]{*} \ovb$ in $\M(A;\Rb^3)$ and
$$\lim_{n \to +\infty} \int_{A \times I} W\left(\nabla_\alpha
u_n\Big|\frac{1}{\e_n} \nabla_3 u_n\right) dx =\ell,$$ for some
$\ell >0$. Then there exist a subsequence $\{n_k\} \nearrow +\infty$
and a sequence $\{v_k\} \subset W^{1,1}(A \times I;\Rb^3)$
satisfying $Tv_k=Tu$ on $\partial A \times I$, $v_k \to u$ in $L^1(A
\times I;\Rb^3), \frac{1}{\e_{n_k}} \int_I \nabla_3 v_k(\cdot,x_3)\,
dx_3 \xrightharpoonup[]{*} \ovb$ in $\M(A;\Rb^3)$, and
$$\limsup_{k \to +\infty} \int_{A \times
I}W\left(\nabla_\alpha v_k\Big|\frac{1}{\e_{n_k}} \nabla_3
v_k\right) dx \leq \ell.$$
\end{Lemma}

\begin{Remark}\label{bdry}
{\rm If $u \in W^{1,1}(\o;\Rb^3)$ then by \cite[Lemma 2.2]{BFM} it
is not necessary to assume neither $\partial A$ to be Lipschitz nor
that $|\ovb|(\partial A)=0$. In that case the conclusion is that
$v_k=u$ on a neighborhood of $\partial A \times I$. }\end{Remark}

To prove the upper bound, we will also need the following locality
result.

\begin{Lemma}\label{lemma2.1bfmbend}
Let $\o \subset \Rb^2$ be a bounded open set with Lipschitz boundary
and let $W:\Rb^{3 \times 3} \to [0,+\infty)$ be a Borel function
satisfying $(H_1)$. For every $(u,\ovb) \in BV(\o;\Rb^3)\times
\M(\o;\Rb^3)$, the set function $J_{\{\e_n\}}(u,\ovb,\cdot)$ is the
trace on $\A(\o)$ of a Radon measure absolutely continuous with
respect to $\LL^2 + |D_\a u| + |\ovb|$.
\end{Lemma}
\begin{proof}
Fix $(u,\ovb) \in BV(\o;\Rb^3)\times \M(\o;\Rb^3)$. Since $\o$ has a
Lipschitz boundary, the extension of $u$ by zero outside $\o$ is a
$BV(\Rb^2;\Rb^3)$. Hence upon extending $u$ and $\ovb$ by zero
outside $\o$, we may assume without loss of generality that $\ovb
\in \M(\Rb^2;\Rb^3)$ and $u \in BV(\Rb^2;\Rb^3)$.

Assume first that $A \in \A_0$, that $\partial A$ is Lipschitz and
satisfies $|\ovb|(\partial A)=0$. By Lemma \ref{1557}, taking
$\{v_n\}$ as test function for $J_{\{\e_n\}}(u,\ovb,A)$ and using
the growth condition $(H_1)$, we get that
$$0 \leq J_{\{\e_n\}}(u,\ovb,A) \leq \b \big(\LL^2(A)
+ |D_\a u|(A) + |\ovb|(A)\big).$$

Consider now an arbitrary open set $A \in \A(\o)$.  By \cite[Example
14.8]{DM}, for any $\eta
>0$, there exists an open set $C$ with smooth boundary such that $A
\subset\subset C$ and
\begin{equation}\label{eta}\LL^2(C \setminus A) + |D_\a u|(C\setminus A)
+ |\ovb|(C\setminus A)<\eta/\b.\end{equation} Note that $C$ may not
be contained in $\o$ and this is the reason why we need to extend
$u$ and $\ovb$ outside $\o$. By \cite[Lemma 14.16]{GT}, the function
$x \mapsto \dist(x,\partial C)$ is smooth on a suitable
$\d$-neighborhood of $\partial C$ for some $\d < \dist(A,\partial
C)$. For every $t \in [0,\d]$, define
$$C_t:=\{x \in C: \dist(x,\partial C)>t\} \quad \text{and}\quad S_t:=\{x \in C: \dist(x,\partial C)=t\}.$$
As the family $\{S_t\}_{t}$ is made of pairwise disjoint sets, it is
possible to find $t_0 \in [0,\d]$ such that $|\ovb|(S_{t_0})=0$.
Since $S_{t_0}=\partial C_{t_0}$, it follows that $C_{t_0}$ is a
smooth open set satisfying $A \subset\subset C_{t_0} \subset C$.
Since $J_{\{\e_n\}}(u,\ovb,\cdot)$ is an increasing set function, we
obtain from the first case together with (\ref{eta}) that
\begin{eqnarray*}J_{\{\e_n\}}(u,\ovb,A) & \leq & J_{\{\e_n\}}(u,\ovb,C_{t_0}) \leq \b
\big(\LL^2(C_{t_0}) + |D_\a
u|(C_{t_0}) + |\ovb|(C_{t_0})\big)\\
& \leq & \b \big(\LL^2(A) + |D_\a u|(A) + |\ovb|(A)\big) +
\eta\end{eqnarray*} and the thesis comes from the arbitrariness of
$\eta$. Repeating word for word the proof of \cite[Lemma 2.1]{BFM},
we get that $J_{\{\e_n\}}(u,\ovb,\cdot)$ is the restriction to
$\A(\o)$ of a Radon measure absolutely continuous with respect to
$\LL^2+|D_\a u| + |\ovb|$. Note that there is no need to extract a
further subsequence as stated in \cite{BFM} since we already did it
passing from $\{\e_j\}$ to $\{\e_n\}\equiv \{\e_{j_n}\}$.
\end{proof}

\section{Proof of the lower bound}

\begin{Lemma}\label{lowerbound}
For every $(u,\ovb) \in BV(\o;\Rb^3) \times \M(\o;\Rb^3)$, then
$J_{\{\e_n\}}(u,\ovb,\o) \geq E(u,\ovb)$.
\end{Lemma}

\begin{proof}
Fix $(u,\ovb, A) \in BV(\omega;\mathbb R^3) \times{\cal
M}(\omega;\mathbb R^3)\times {\cal A}(\omega)$. Thanks to the
Besicovitch Decomposition Theorem, one may find four mutually
singular measures $\ovb^a$, $\ovb^j$, $\ovb^c$ and $\ovb^\sigma$
such that $\ovb=\ovb^a+\ovb^j+\ovb^c+\ovb^\sigma$ and $\ovb^a \ll
\LL^2$, $\ovb^j \ll \HH^1 \res\, J_u$ and $\ovb^c \ll |D^c_\a u|$.

Consider a sequence $\{u_n\} \subset W^{1,1}(\O;\Rb^3)$ such that
$u_n \xrightharpoonup[]{*} u$ in $BV(\O;\mathbb R^3)$,
$\frac{1}{\e_n} \int_I \nabla_3 u_n(\cdot,x_3)\, dx_3
\xrightharpoonup[]{*} \ovb$ in ${\cal M}(\omega;\mathbb R^3)$, and
$$J_{\{\e_n\}}(u,\ovb,\o)=\lim_{n \to +\infty} \int_{\O}W\left(\nabla_\a u_n\Big|\frac{1}{\e_n} \nabla_3 u_n\right) dx.$$
For every Borel set $B \subset \o$, define
$$\mu_n(B):= \int_{B \times I} W\left(\nabla_\a u_n\Big|\frac{1}{\e_n} \nabla_3 u_n\right) dx \quad
\text{ and }\quad \overline b_n(B):=\frac{1}{\e_n}\int_{B \times
I}\nabla_3 u_n\ dx.$$ It turns out that $\{\mu_n\}$ and
$\{|\overline b_n|\}$ are sequences of nonnegative Radon measures
uniformly bounded in ${\cal M}(\o)$. Hence we can extract
subsequences, still denoted $\{\mu_n\}$ and $\{|\overline b_n|\}$,
and find $\mu$ and $\lambda \in {\cal M}(\o)$ so that $\mu_n
\xrightharpoonup[]{*} \mu$ and $|\overline b_n|\xrightharpoonup[]{*}
\lambda$ in $\M(\o)$. Similarly we can decompose the measure $\mu$
as the sum of five mutually singular measures $\mu^a$, $\mu^j$,
$\mu^c$, $\mu^\sigma$ and $\mu^s$ such that
$\mu=\mu^a+\mu^j+\mu^c+\mu^\sigma+\mu^s$ and $\mu^a \ll \LL^2$,
$\mu^j \ll \HH^1 \res\, J_u$, $\mu^c \ll |D^c u|$ and $\mu^\sigma
\ll |\ovb^\sigma|$.

Since $\mu(\o) \leq J_{\{\e_n\}}(u,\ovb,\o)$, in order to show the
lower bound, it is enough to check that $\mu(\o)\geq E(u,\ovb)$ or
that
\begin{eqnarray}
\frac{d\mu^a}{d\LL^2}(x_0) \geq \mathcal Q^*W\left(\nabla_\a
u(x_0)\Big|\frac{d\ovb}{d\LL^2}(x_0)\right) \quad \text{ for
$\LL^2$-a.e. }x_0 \in \o,\label{mua}\\
\frac{d\mu^j}{d\HH^1\res\, J_u}(x_0) \geq (\mathcal Q^\ast
W)^\infty\left((u^+(x_0)-u^-(x_0))\otimes
\nu_u(x_0),\frac{d\ovb}{d\HH^1\res\,
J_u}(x_0)\right) \quad \text{ for $\HH^1$-a.e. }x_0 \in J_u,\label{muj}\\
\frac{d\mu^c}{d|D_\a^c u|}(x_0) \geq (\mathcal
Q^*W)^\infty\left(\frac{dD_\a u}{d|D_\a^c
u|}(x_0)\Big|\frac{d\ovb}{d|D_\a^c u|}(x_0)\right) \quad \text{ for
$|D_\a^c u|$-a.e. }x_0 \in \o,\label{muc}\\
\frac{d\mu^\sigma}{d|\ovb^\sigma|}(x_0) \geq (\mathcal
Q^*W)^\infty\left(0\Big|\frac{d\ovb}{d|\ovb^\sigma|}(x_0)\right)
\quad \text{ for $|\ovb^\sigma|$-a.e. }x_0 \in \o.\label{mus}
\end{eqnarray}
Indeed, if the four previous properties hold, we obtain that
\begin{eqnarray*}
&&\int_\o \mathcal Q^*W\left(\nabla_\a
u\Big|\frac{d\ovb}{d\LL^2}\right) dx + \int_{J_u}(\mathcal Q^\ast
W)^\infty\left((u^+-u^-)\otimes \nu_u,\frac{d\ovb}{d\HH^1\res\,
J_u}\right)d\HH^1\\
&&+\int_\o (\mathcal Q^*W)^\infty\left(\frac{dD_\a u}{d|D_\a^c
u|}\Big|\frac{d\ovb}{d|D_\a^c u|}\right)d|D_\a^c u|+
\int_\o(\mathcal
Q^*W)^\infty\left(0\Big|\frac{d\ovb}{d|\ovb^\sigma|}\right)d|\ovb^\sigma|\\
&& \hspace{1.5cm}=\mu^a(\o)+\mu^j(\o)+\mu^c(\o)+\mu^\sigma(\o) \leq
\mu(\o)\leq J_{\{\e_n\}}(u,\ovb,\o),
\end{eqnarray*}
which is the announced claim.
\end{proof}

The remaining of the section is devoted to prove the inequalities
(\ref{mua})-(\ref{mus})

\vskip5pt

\noindent {\bf Proof of (\ref{mua}).} Let $x_0 \in \o$ be such that
the Radon-Nikod\'ym derivative of $\mu$ and $\ovb$ at $x_0$ with
respect to $\LL^2$ exist and are finite, which is also a Lebesgue
point for $u$, $\nabla_\a u$ and $\frac{d\ovb}{d\LL^2}$, a point of
approximate differentiability of $u$, and
\begin{equation}\label{ba}
\frac{d|\mu-\mu^a|}{d\LL^2}(x_0)=\frac{d|\ovb-\ovb^a|}{d\LL^2}(x_0)=0.
\end{equation}
Note that since $|\ovb-\ovb^a|$ and $|\mu-\mu^a|$ are singular with
respect to the Lebesgue measure, then $\LL^2$ almost every points
$x_0 \in \o$ satisfy these properties.  Let $\{\rho_k\}$ be a
sequence converging to zero and such that $\lambda(\partial
Q'(x_0,\rho_k))=\mu(\partial Q'(x_0,\rho_k))=0$ for every $k \in
\Nb$. Hence it follows from (\ref{ba}) that
\begin{eqnarray}
\frac{d\mu^a}{d\LL^2}(x_0) &=& \frac{d\mu}{d\LL^2}(x_0)  = \lim_{k \to +\infty}\frac{\mu(Q'(x_0,\rho_k))}{\rho_k^2}\nonumber\\
& = &\lim_{k \to +\infty} \lim_{n \to +\infty}\frac{1}{\rho_k^2}
\int_{Q'(x_0,\rho_k) \times I} W\left(\nabla_\a u_n \Big|\frac{1}{\e_n} \nabla_3 u_n\right) dx\nonumber\\
& = & \lim_{k \to +\infty} \lim_{n \to +\infty} \int_{Q' \times I}
W\left(\nabla_\a u_n(x_0 +\rho_k y_\alpha, y_3)\Big|
\frac{1}{\e_n} \nabla_3 u_n(x_0+\rho_k y_\alpha, y_3)\right) dy\nonumber\\
& = & \lim_{k \to +\infty} \lim_{n \to +\infty} \int_{Q ' \times I}
W\left(\nabla_\a u_{n,k}\Big|\frac{\rho_k}{\e_n} \nabla_3
u_{n,k}\right)dy,\label{diag1}
\end{eqnarray}
where we set $u_{n,k}(y_\a,y_3):= [u_n(x_0+\rho_k y_\alpha,
y_3)-u(x_0)]/\rho_k$.

Since $x_0$ is a point of approximate differentiability of $u$ and
$u_n \to u$ in $L^1(\O;\Rb^3)$, defining $u_0(y_\a,y_3):=\nabla_\a
u(x_0) y_\a$, it results that
\begin{equation}\label{approxuen}
\lim_{k \to +\infty} \lim_{n \to +\infty} \|u_{n,k}-u_0 \|_{L^1(Q'
\times I;\mathbb R^3)}=0.
\end{equation}
On the other hand, using (\ref{ba}), the fact that $(1/\e_n) \int_I
\nabla_3 u_n(\cdot,x_3)\, dx_3 \xrightharpoonup[]{*} \ovb$ in
$\M(\o;\Rb^3)$ and that $x_0$ is a Lebesgue point of
$\frac{d\ovb}{d\LL^2}$, for every $\varphi \in \C_0(Q';\Rb^3)$ we
get that
\begin{equation}\label{diag3}
\lim_{k \to +\infty} \lim_{n \to +\infty} \int_{Q'} \left(
\frac{\rho_k}{\e_n} \int_I \nabla_3 u_{n,k}(y_\a,y_3)\, dy_3 \right)
\cdot \varphi(y_\a)\, dy_\a= \frac{d\ovb}{d\LL^2}(x_0)\cdot
\int_{Q'}\varphi(y_\a)\, dy_\a.
\end{equation}
Moreover, since $\lambda(\partial Q'(x_0,\rho_k))=0$ for each $k \in
\Nb$, \cite[Proposition 1.62]{AFP} ensures that $\overline
b_n(Q'(x_0,\rho_k)) \to \overline b(Q'(x_0,\rho_k))$ as $n \to
+\infty$ and thus
\begin{equation}\label{diag4bbis}
\lim_{k \to +\infty} \lim_{n \to +\infty} \frac{\rho_k}{\e_n}
\int_{Q'\times I} \nabla_3 u_{n,k}\, dy= \frac{d\ovb}{d\LL^2}(x_0).
\end{equation}
Gathering (\ref{diag1}), (\ref{approxuen}), (\ref{diag3}) and
(\ref{diag4bbis}) and using the fact that $\M(Q';\Rb^3)$ is the dual
of the separable space $\C_0(Q';\Rb^3)$, by means of a standard
diagonalization process, one may construct a sequence $\bar
u_k:=u_{n_k,k}-u_0$ and $\d_k:=\e_{n_k}/\rho_k$ such that $\bar u_k
\to 0$ in $L^1(Q' \times I;\Rb^3)$, $\d_k \to 0$, $(1/\d_k) \int_I
\nabla_3 \bar u_k(\cdot,y_3)\, dy_3 \xrightharpoonup[]{*}
\frac{d\ovb}{d\LL^2}(x_0)\LL^2$ in $\M(Q';\Rb^3)$,
\begin{equation}\label{diag4bis}
\frac{1}{\d_k} \int_{Q'\times I} \nabla_3 \bar u_k\, dy \to
\frac{d\ovb}{d\LL^2}(x_0).
\end{equation} and
\begin{equation}\label{3.13bfm}
\frac{d\mu^a}{d\LL^2}(x_0)=\lim_{k \to +\infty} \int_{Q' \times I}
W\left(\nabla _\a u(x_0)+\nabla_\a \bar u_k\Big|
\frac{1}{\d_k}\nabla_3 \bar u_k\right) dy.
\end{equation}
Using Remark \ref{bdry}, one may assume without loss of generality
that $\bar u_k=0$ on a neighborhood of $\partial Q' \times I$. We
now define
$$\varphi_k(x_\a,x_3):=\bar u_k(x_\a,x_3) +\d_k x_3
\left(\frac{d\ovb}{d\LL^2}(x_0) - \frac{1}{\d_k} \int_{Q' \times
I}\nabla_3 \bar u_k(y)\, dy\right).$$ Then, $\varphi_k \in
W^{1,1}(Q' \times I;\Rb^3)$, $\varphi_k(\cdot,x_3)$ is $Q'$-periodic
for $\LL^1$-a.e. $x_3 \in I$ and
$$\frac{1}{\d_k} \int_{Q' \times I} \nabla_3\varphi_k\, dy =
\frac{d\ovb}{d\LL^2}(x_0).$$ Hence $\varphi_k$ is an admissible test
function for $\mathcal Q^*W\Big(\nabla_\a
u(x_0)\big|\frac{d\ovb}{d\LL^2}(x_0)\Big)$, and using
(\ref{3.13bfm}) together with the Lipschitz property (\ref{Wlip}),
we get that
\begin{eqnarray*}
\frac{d\mu^a}{d\LL^2}(x_0) & \geq & \limsup_{k \to +\infty} \int_{Q'
\times I} W\left(\nabla _\a u(x_0)+\nabla_\a \varphi_k\Big|
\frac{1}{\d_k} \nabla_3 \varphi_k\right) dy\\
&& - L \limsup_{k \to +\infty} \left|\frac{d\ovb}{d\LL^2}(x_0) -
\frac{1}{\d_k} \int_{Q' \times I}\nabla_3 \bar u_k\,
dy\right|.\end{eqnarray*}  Relation (\ref{diag4bis}) enables us
to conclude that the last term in the previous inequality is
actually zero and thus
$$\frac{d\mu^a}{d\LL^2}(x_0) \geq \mathcal Q^*W\left(\nabla_\a
u(x_0)\Big|\frac{d\ovb}{d\LL^2}(x_0)\right).$$

\vskip5pt

\noindent {\bf Proof of (\ref{muj}).} Let $x_0 \in J_u$, then there
exists $u^-(x_0)$, $u^+(x_0) \in \Rb^3$ (with $u^-(x_0) \neq
u^+(x_0)$) and $\nu=\nu_u(x_0) \in \mathbb S^1$ such that
$$\lim_{\rho \to 0^+}\frac{1}{\rho^2}\int_{\{y_\a \in Q'_\nu(x_0,\rho): \, \pm (y_\a-x_0)\cdot
\nu>0\}}|u(y_\a)-u^\pm(x_0)|\, dy_\a =0.$$ Assume that the
Radon-Nikod\'ym derivative of $\mu$ and $\ovb$ at $x_0$ with respect
to $\HH^1\res \, J_u$ exist and are finite, that $x_0$ is Lebesgue
point for $\frac{d\ovb}{d\HH^1\res\, J_u}$ with respect to $\HH^1
\res \, J_u$, that
\begin{equation}\label{bj}\frac{d|\mu-\mu^j|}{d\HH^1 \res\,
J_u}(x_0)=\frac{d|\ovb-\ovb^j|}{d\HH^1 \res\,
J_u}(x_0)=0,\end{equation} and
\begin{equation}\label{rect}\lim_{\rho \to 0^+} \frac{\HH^1(J_u \cap
Q_\nu'(x_0,\rho))}{\rho}=1.\end{equation} Assume further that
$\pi_\nu:=\nu^\perp$ is the approximate tangent space of $J_u$ at
$x_0$, {\it i.e.}, \begin{equation}\label{approx}\lim_{\rho \to
0}\frac{1}{\rho}\int_{J_u}\phi\left(\frac{x_\a-x_0}{\rho}\right)\,
d\HH^1(x_\a)=\int_{\pi_\nu}\phi(x_\a)\, d\HH^1(x_\a)\quad \text{ for
all }\phi \in \C_c(\Rb^2).\end{equation} Note that $\HH^1$ almost
every points $x_0$ in $J_u$ satisfy the preceding requirements.
Indeed (\ref{rect}) is a consequence of the countably
$\HH^1$-rectifiability of $J_u$ (see {\it e.g.} \cite[Theorem
2.63]{AFP}), property (\ref{bj}) is due to the fact that the
measures $|\mu-\mu^j|$ and $|\ovb-\ovb^j|$ are singular with respect
to $\HH^1\res\, J_u$ while (\ref{approx}) is a consequence of the
Federer-Vol'pert Theorem (see \cite[Theorem 3.78]{AFP}).

Let $\{\rho_k\} \searrow 0^+$ be such that $\lambda(\partial
Q'_\nu(x_0, \rho_k))=\mu(\partial Q'_\nu(x_0,\rho_k))=0$ for each $k
\in \Nb$. Then by virtue of (\ref{bj}) and (\ref{rect}), we infer
that
\begin{eqnarray}
\frac{d\mu^j}{d\HH^1 \res \, J_u}(x_0) & = & \frac{d\mu}{d\HH^1 \res
\, J_u}(x_0)=\lim_{k \to +\infty}\frac{\mu(Q'_\nu(x_0,\rho_k))}
{\HH^1(Q'_\nu(x_0,\rho_k) \cap J_u)}=\lim_{k \to
+\infty}\frac{\mu(Q'_\nu(x_0,\rho_k))}
{\rho_k}\nonumber\\
& = &\lim_{k \to +\infty} \lim_{n \to +\infty}\frac{1}{\rho_k}
\int_{Q'_\nu(x_0,\rho_k) \times I} W\left(\nabla_\a u_n \Big|\frac{1}{\e_n} \nabla_3 u_n\right) dx\nonumber\\
& = & \lim_{k \to +\infty} \lim_{n \to +\infty} \rho_k \int_{Q'_\nu
\times I} W\left(\nabla_\a u_n(x_0 +\rho_k y_\alpha, y_3)\Big|
\frac{1}{\e_n} \nabla_3 u_n(x_0+\rho_k y_\alpha, y_3)\right) dy\nonumber\\
& = & \lim_{k \to +\infty} \lim_{n \to +\infty} \rho_k \int_{Q'_\nu
\times I} W\left(\frac{1}{\rho_k} \left(\nabla_\a
v_{n,k}\Big|\frac{\rho_k}{\e_n} \nabla_3 v_{n,k}\right)
\right)dy,\label{diag4}
\end{eqnarray}
where $v_{n,k}(y):= u_n(x_0+\rho_k y_\alpha,y_3)$. Set
$$v_0(y):=\left\{
\begin{array}{ll}
u^+(x_0) & \hbox{ if } y_\a \cdot \nu>0\\
u^-(x_0) & \hbox{ if } y_\a \cdot \nu\leq 0.
\end{array}
\right.$$ As $x_0 \in J_u$ and $u_n \to u$ in $L^1(\O;\Rb^3)$, it
follows that
\begin{equation}\label{diag5}
\lim_{k \to +\infty} \lim_{n \to +\infty}\|v_{n,k}-
v_0\|_{L^1(Q'_\nu \times I;\Rb^3)} =0.
\end{equation}
 Since $(1/\e_n)\int_I \nabla_3 u_n(\cdot,x_3)\, dx_3
\xrightharpoonup[]{*} \ovb$ in $\M(\o;\mathbb R^3)$, we deduce that
for every $\varphi \in \C_0(Q'_\nu;\Rb^3)$, $$\lim_{n \to +\infty}
\int_{Q'_\nu} \left( \frac{\rho_k}{\e_n} \int_I \nabla_3
v_{n,k}(y_\a,y_3)\, dy_3 \right) \cdot \varphi(y_\a)\, dy_\a=
\frac{1}{\rho_k} \int_{Q'_\nu(x_0,\rho_k)}
\varphi\left(\frac{x_\a-x_0}{\rho_k}\right)\cdot d\overline
b(x_\a).$$ Using the fact that $x_0$ is a Lebesgue point of
$\frac{d\ovb}{d\HH^1\res\, J_u}$ together with (\ref{bj}),
(\ref{rect}) and (\ref{approx}) we infer that
\begin{eqnarray}\label{diag6}\lim_{k \to +\infty}\lim_{n \to +\infty}
\int_{Q'_\nu} \left( \frac{\rho_k}{\e_n} \int_I \nabla_3
v_{n,k}(y_\a,y_3)\, dy_3 \right) \cdot \varphi(y_\a)\, dy_\a & = &
\frac{d\ovb}{d\HH^1\res\, J_u}(x_0)\cdot \int_{\pi_\nu}\varphi\,
d\HH^1.
\end{eqnarray}
Moreover, as for the bulk part, using the fact that
$\lambda(\partial Q'_\nu(x_0,\rho_k))=0$ for every $k \in \Nb$, we
have that
\begin{equation}\label{diag7}
\lim_{k \to +\infty} \lim_{n \to +\infty} \frac{\rho_k}{\e_n}
\int_{Q'_\nu \times I} \nabla_3 v_{n,k}\, dy=
\frac{d\ovb}{d\HH^1\res\, J_u}(x_0).
\end{equation}
Using again the separability of $\C_0(Q'_\nu;\Rb^3)$ together with a
diagonalization argument, from (\ref{diag4}), (\ref{diag5}),
(\ref{diag6}) and (\ref{diag7}), we obtain the existence of
sequences $\bar v_k:=v_{n_k,k} \in W^{1,1}(Q'_\nu \times I;\Rb^3)$
and $\d_k:=\e_{n_k}/\rho_k$ with the properties that $\d_k \to 0$,
$\bar v_k \to v_0$ in $L^1(Q'_\nu \times I;\Rb^3)$, $(1/\d_k) \int_I
\nabla_3 \bar v_k(\cdot,x_3)\, dx_3 \xrightharpoonup[]{*}
\frac{d\ovb}{d\HH^1\res\, J_u}(x_0)\HH^1\res\, \pi_\nu$ in
$\M(Q'_\nu;\Rb^3)$,
\begin{equation}\label{diag7bis}
\frac{1}{\d_k}\int_{Q'_\nu \times I} \nabla_3 \bar v_k\, dy \to
\frac{d\ovb}{d\HH^1\res\, J_u}(x_0),
\end{equation}
and
$$\frac{d\mu^j}{d\HH^1 \res \, J_u}(x_0) =\lim_{k \to +\infty} \rho_k \int_{Q'_\nu
\times I} W\left(\frac{1}{\rho_k} \left(\nabla_\a \bar v_k \Big|
\frac{1}{\d_k} \nabla_3 \bar v_k\right) \right)dy.$$ By the
coercivity condition $(H_1)$ and the previous relation, it follows
that the sequence of scaled gradients $\{(\nabla_\a \bar
v_k|(1/\d_k) \nabla_3 \bar v_k)\}$ is uniformly bounded in
$L^1(Q'_\nu \times I;\Rb^{3 \times 3})$. Thus, using $(H_2)$ and the
fact that the recession function $W^\infty$ is positively
$1$-homogeneous, we obtain that
\begin{eqnarray*}
&&\rho_k \int_{Q'_\nu \times I} \left| W\left(\frac{1}{\rho_k}
\left(\nabla_\a \bar v_k\Big| \frac{1}{\d_k} \nabla_3 \bar
v_k\right) \right) - W^\infty\left(\frac{1}{\rho_k} \left(\nabla_\a
\bar v_k\Big| \frac{1}{\d_k}
\nabla_3 \bar v_k\right) \right) \right| dy\\
&& \hspace{1cm} \leq  C \int_{Q'_\nu \times I} \left(\rho_k +
\rho_k^r \left|\left(\nabla_\a \bar v_k \Big|\frac{1}{\d_k} \nabla_3 \bar v_k\right)\right|^{1-r}\right)\, dy\\
&& \hspace{1cm} \leq  C\rho_k + C\rho_k^r \|(\nabla_\a \bar v_k
|(1/\d_k) \nabla_3 \bar v_k)\|^{1-r}_{L^1(Q'_\nu \times I;\Rb^{3
\times 3})} \to 0
\end{eqnarray*}
where we applied H\"older's Inequality. As a consequence
$$\frac{d\mu^j}{d\HH^1 \res \, J_u}(x_0) =  \lim_{k \to +\infty}
\int_{Q'_\nu \times I} W^\infty \left(\nabla_\a \bar
v_k\Big|\frac{1}{\d_k} \nabla_3 \bar v_k\right)dy.$$

Since $\HH^1(\pi_\nu \cap \partial Q'_\nu)=0$, we can apply Lemma
\ref{tracegamma} (with $W^\infty$ instead of $W$) so that, up to an
extraction, there is no loss of generality to assume that $T\bar
v_k=Tv_0$. Define
$$\varphi_k(x_\a,x_3):=\bar v_k(x_\a,x_3) +\d_k x_3
\left(\frac{d\ovb}{d\HH^1\res\, J_u}(x_0) - \frac{1}{\d_k}
\int_{Q'_\nu \times I}\nabla_3 \bar v_k(y)\, dy \right),$$ and
denote by $\varphi_k^{\pm \nu}$ the trace of $\varphi_k$ on the
faces $\{(x_\a,x_3) \in Q'_\nu \times I : x_\a \cdot \nu = \pm
1/2\}$. Then $\varphi_k \in W^{1,1}(Q'_\nu \times I;\Rb^3)$ is
$1$-periodic in the direction $\tau$ (where $\tau \in \mathbb S^1$
is such that $(\tau,\nu)$ is an orthonormal basis of $\Rb^2$),
$\varphi_k^{+\nu}-\varphi_k^{-\nu}=u^+(x_0)-u^-(x_0)$ and
$$\frac{1}{\d_k} \int_{Q'_\nu \times I} \nabla_3 \varphi_k\, dy = \frac{d\ovb}{d\HH^1\res\,
J_u}(x_0).$$ In particular, $\varphi_k$ is an admissible test
function for
$\g\Big(u^+(x_0)-u^-(x_0),\nu_u(x_0),\frac{d\ovb}{d\HH^1\res\,
J_u}(x_0)\Big)$ and using the Lipschitz condition (\ref{Winftylip})
satisfied by $W^\infty$, we infer that
\begin{eqnarray*}\frac{d\mu^j}{d\HH^1 \res \, J_u}(x_0) & \geq & \limsup_{k \to
+\infty} \int_{Q'_\nu \times I} W^\infty \left(\nabla_\a
\varphi_k\Big| \frac{1}{\d_k} \nabla_3 \varphi_k\right) dy\\
&& - L \limsup_{k \to +\infty}\left|\frac{d\ovb}{d\HH^1\res\,
J_u}(x_0) - \frac{1}{\d_k} \int_{Q'_\nu \times I}\nabla_3 \bar v_k\,
dy \right|.\end{eqnarray*}  From (\ref{diag7bis}), it follows
that the last term of the previous relation is actually zero. Hence
$$\frac{d\mu^j}{d\HH^1 \res \, J_u}(x_0) \geq
\g\left(u^+(x_0)-u^-(x_0),\nu_u(x_0),\frac{d\ovb}{d\HH^1\res\,
J_u}(x_0)\right),$$ and consequently by virtue of (\ref{Wlip}) and
Proposition \ref{denssurf} it results that
$$\frac{d\mu^j}{d\HH^1 \res \, J_u}(x_0) \geq (\mathcal Q^\ast
W)^{\infty}\left((u^+(x_0)-u^-(x_0)) \otimes
\nu_u(x_0),\frac{d\ovb}{d\HH^1\res\, J_u}(x_0)\right).$$

\vskip5pt

\noindent {\bf Proof of (\ref{muc}).} Fix a point $x_0 \in \o$ such
that the matrix
\begin{equation}\label{cantor2}
A(x_0):=\frac{dD_\a u}{d|D_\a u|}(x_0) \text{ has rank one and }
|A(x_0)|=1,
\end{equation}
the Radon-Nikod\'ym derivative of $\mu$ and $\ovb$ with respect to
$|D_\a^c u|$ exist and are finite,
\begin{equation}\label{cantor3}
\frac{d|\mu-\mu^c|}{d|D_\a^c u|}(x_0)=
\frac{d|\ovb-\ovb^c|}{d|D_\a^c u|}(x_0)=0,
\end{equation}
\begin{equation}\label{cantor4}
\frac{d|D_\a u|}{d|D_\a^c u|}(x_0)=1
\end{equation}
and
\begin{equation}\label{cantor5}
\lim_{\rho \to 0^+} \frac{|D_\a u|(Q'(x_0,\rho))}{\rho}=0, \quad
\lim_{\rho \to 0^+} \frac{|D_\a u|(Q'(x_0,\rho))}{\rho^2}=+\infty.
\end{equation}
Assume also that for every $t \in (0,1)$,
\begin{equation}\label{cantor6}
\liminf_{\rho \to 0^+} \frac{|D_\a u|(Q'(x_0,\rho) \setminus
Q'(x_0,t\rho))}{|D_\a u|(Q'(x_0,\rho))} \leq 1 - t^2.
\end{equation}
Note that $|D_\a^c u|$ almost every points $x_0$ in $\o$ satisfy
these properties. Indeed, (\ref{cantor2}) is a consequence of
Alberti's Rank One Theorem (see \cite{A}); properties
(\ref{cantor3}) come from the fact that $|\mu-\mu^c|$ and
$|\ovb-\ovb^c|$ are singular with respect to $|D_\a^c u|$; property
(\ref{cantor4}) is due to the Besicovitch Differentiation Theorem;
both relations in (\ref{cantor5}) are obtained from
\cite[Proposition 3.92]{AFP} and finally, property (\ref{cantor6})
is proved in \cite[Lemma 2.13]{FM}.

Since $A(x_0)$ has rank one, there exists $a \in \Rb^3$ and $\nu \in
\mathbb S^1$ such that $A(x_0):=a \otimes \nu$. We may assume
without loss of generality that $\nu=e_2$.

Fix $t \in (0,1)$ arbitrarily close to $1$ and thanks to
(\ref{cantor6}), choose a sequence $\{\rho_k\} \searrow 0^+$ such
that
\begin{equation}\label{cantor7}
\limsup_{k \to +\infty} \frac{|D_\a u|(Q'(x_0,\rho_k) \setminus
Q'(x_0,t\rho_k))}{|D_\a u|(Q'(x_0,\rho_k))} \leq 1 - t^2.
\end{equation}
Moreover, up to a subsequence (still denoted $\{\rho_k\}$), it is
possible to find $\nu^c \in {\rm Tan}(|D^c_\a u|,x_0)$ (depending on
$t$), {\it i.e.},
\begin{equation}\label{cantor6bis} \lim_{k \to +\infty}
\frac{1}{|D^c_\a u|(Q'(x_0,\rho_k))}
\int_{\Rb^2}\varphi\left(\frac{x_\a-x_0}{\rho_k} \right)\, d|D^c_\a
u|(x_\a)=\int_{\Rb^2} \varphi(x_\a)\, d\nu^c(x_\a),
\end{equation}
for all $\varphi \in \C_c(\Rb^2)$. Fix $\g \in (t,1)$, then by
(\ref{cantor3}) and (\ref{cantor4}),
\begin{eqnarray}\label{eqc2}
\frac{d\mu^c}{d|D_\a^c u|}(x_0) & = & \frac{d\mu}{d|D_\a^c
u|}(x_0)=\lim_{k \to +\infty}\frac{\mu(Q'(x_0,\rho_k))}{|D_\a^c
u|(Q'(x_0,\rho_k))}=\lim_{k \to
+\infty}\frac{\mu(Q'(x_0,\rho_k))}{|D_\a u|(Q'(x_0,\rho_k))}\nonumber\\
& \geq &\limsup_{k \to +\infty}\limsup_{n \to +\infty}
\frac{1}{|D_\a u|(Q'(x_0,\rho_k))}\int_{Q'(x_0,\g \rho_k) \times I}
W\left(\nabla_\a u_n\Big|\frac{1}{\e_n} \nabla_3 u_n \right)\, dx.
\end{eqnarray}
Define
$$\left\{
\begin{array}{l}
\ds z_k(x_\a):=\frac{\rho_k}{|D_\a u|(Q'(x_0,\rho_k))}\left[
u(x_0+\rho_k x_\a)
- \int_{Q'} u(x_0+\rho_k y_\a)\, dy_\a \right],\\
\ds w_{n,k}(x_\a,x_3):=\frac{\rho_k}{|D_\a u|(Q'(x_0,\rho_k))}\left[
u_n (x_0+\rho_kx_\a,x_3) - \int_{Q' \times I}
u_n(x_0+\rho_ky_\a,y_3)\, dy \right].
\end{array}
\right.$$ Changing variable in (\ref{eqc2}) and setting
$$t_k:=\frac{|D_\a u|(Q'(x_0,\rho_k))}{\rho_k^2},$$ we get that
\begin{equation}\label{eqc1}\frac{d\mu^c}{d|D_\a^c u|}(x_0)\geq \limsup_{k \to +\infty}\limsup_{n \to
+\infty} \frac{1}{t_k}\int_{(\g Q') \times I} W\left(t_k
\left(\nabla_\a w_{n,k} \Big|\frac{\rho_k}{\e_n} \nabla_3
w_{n,k}\right)\right) dx.\end{equation} Using the fact that $u_n \to
u$ in $L^1(\O;\Rb^3)$ we obtain
\begin{equation}\label{eqc3}
\lim_{k \to +\infty}\lim_{n \to +\infty}\|w_{n,k}-z_k\|_{L^1(Q'
\times I;\Rb^3)}=0.\end{equation} As $\int_{Q'}z_k\, dx_\a =0$ and
$|D_\a z_k|(Q')=1$, it follows that the sequence $\{z_k\}$ is
relatively compact in $L^1(Q';\Rb^3)$ and by \cite[Theorem
3.95]{AFP}, any limit function $w$ is representable by
$$w(x_\a)=a\, \theta(x_2)$$
for some increasing function $\theta \in BV(-1/2,1/2)$ (recall that
we assumed $\nu=e_2$). Hence, using (\ref{eqc3}) it follows that
\begin{equation}\label{eqc4}
\lim_{k \to +\infty}\lim_{n \to +\infty}\|w_{n,k}-w\|_{L^1(Q' \times
I;\Rb^3)}=0.
\end{equation}
Now take $\varphi \in \C_0(Q';\Rb^3)$, then changing variables %it
using the fact that $(1/\e_n)\int_I\nabla_3 u_n(\cdot,y_3)\, dy_3
\xrightharpoonup[]{*} \ovb$ in $\M(\o;\Rb^3)$ together with
(\ref{cantor3}), (\ref{cantor4}) and (\ref{cantor6bis}), it follows
that
\begin{equation}\label{eqc5} \lim_{k \to +\infty}\lim_{n \to
+\infty}\int_{Q'}\varphi(x_\a) \cdot \left(\frac{\rho_k}{\e_n}
\int_I \nabla_3 w_{n,k}(x_\a,x_3)\, dx_3 \right) dx_\a =
\frac{d\ovb}{d|D_\a^c u|}(x_0) \cdot \int_{Q'}\varphi(x_\a)\,
d\nu^c(x_\a).
\end{equation} Gathering (\ref{eqc1}), (\ref{eqc4}) and
(\ref{eqc5}), the separability of $\C_0(Q';\Rb^3)$ together with a
standard diagonalization argument, it leads to the existence of a
subsequence $n_k \nearrow +\infty$ such that, setting $\bar
w_k:=w_{n_k,k}$ and $\d_k:=\e_{n_k}/\rho_k$, then $\d_k \searrow
0^+$, $\bar w_k \to w$ in $L^1(Q' \times I;\Rb^3)$, $\frac{1}{\d_k}
\int_I\nabla_3 \bar w_k(\cdot,x_3)\, dx_3 \xrightharpoonup[]{*}
\frac{d\ovb}{d|D_\a^c u|}(x_0)\nu^c$ in $\M(Q';\Rb^3)$ and
\begin{equation}\label{eqc6}\frac{d\mu^c}{d|D_\a^c u|}(x_0) \geq \limsup_{k \to +\infty}
\frac{1}{t_k}\int_{(\g Q') \times I} W\left(t_k \left(\nabla_\a \bar
w_k \Big|\frac{1}{\d_k} \nabla_3 \bar w_k\right)\right)
dx.\end{equation} We may also assume without loss of generality that
$$\left|\frac{1}{\d_k} \int_I\nabla_3 \bar w_k(\cdot,x_3)\,
dx_3\right| \xrightharpoonup[]{*} \lambda^c \quad \text{ in
}\M(Q')$$ for some non negative Radon measure $\lambda^c \in
\M(Q')$.

Thanks to the coercivity condition $(H_1)$, the sequence of scaled
gradients $\{(\nabla_\a \bar w_k|(1/\d_k) \nabla_3 \bar w_k)\}$ is
uniformly bounded in $L^1((\g Q') \times I;\Rb^{3 \times 3})$. Thus
using hypothesis $(H_2)$ and H\"older's Inequality, we get that
\begin{eqnarray*}
&&\frac{1}{t_k}\int_{(\g Q') \times I}\left|W^\infty\left(t_k
\left(\nabla_\a \bar w_k \Big|\frac{1}{\d_k} \nabla_3 \bar
w_k\right)\right) - W\left(t_k \left(\nabla_\a \bar w_k
\Big|\frac{1}{\d_k}
\nabla_3 \bar w_k\right)\right)\right|\, dx\\
&&\hspace{2cm} \leq \frac{C}{t_k} +\frac{C}{t_k^r}\int_{(\g Q')
\times I}\left|\left(\nabla_\a \bar w_k \Big|\frac{1}{\d_k} \nabla_3 \bar w_k\right)\right|^{1-r}\, dx\\
&&\hspace{2cm} \leq \frac{C}{t_k}
+\frac{C}{t_k^r}\left\|\left(\nabla_\a \bar w_k |(1/\d_k) \nabla_3
\bar w_k\right)\right\|_{L^1((\g Q') \times I;\Rb^{3 \times
3})}^{1-r} \to 0,
\end{eqnarray*}
where we used the fact that, thanks to (\ref{cantor5}), $t_k  \to
+\infty$. But as $W^\infty$ is positively $1$-homogeneous, we get
from (\ref{eqc6}) $$\frac{d\mu^c}{d|D_\a^c u|}(x_0)\geq \limsup_{k
\to +\infty} \int_{(\g Q') \times I} W^\infty\left(\nabla_\a \bar
w_k \Big|\frac{1}{\d_k} \nabla_3 \bar w_k\right) dx.$$

Extend $\theta$ continuously to $\Rb$ by the value of its trace at
$\pm1/2$. Let $\varrho_k$ be a usual sequence of (one dimensional)
mollifiers and set
$$\tilde w_k(x_\a,x_3):= a (\theta * \varrho_k)(x_2) + \d_k x_3
\frac{d\ovb}{d|D_\a^c u|}(x_0).$$ Then $\tilde w_k \in W^{1,1}(Q'
\times I;\Rb^3)$, $\tilde w_k \to w$ in $L^1(Q' \times I;\Rb^3)$ and
$\frac{1}{\d_k} \int_I \nabla_3 \tilde w_k\,
dx_3=\frac{d\ovb}{d|D_\a^c u|}(x_0)$ for each $k \in \Nb$. Thus
$z_k-\tilde w_k \to 0$ in $L^1(Q' \times I;\Rb^3)$ and
\begin{equation}\label{eqc8} D_\a z_k((sQ') \times I) - D_\a \tilde
w_k((sQ') \times I) \to 0\end{equation} except at most for countably
many $s \in (0,1)$. Fix $s \in (t,\g)$ so that (\ref{eqc8}) holds
and $\lambda^c(\partial(sQ'))=0$. Using a standard cut-off function
argument, we may assume without loss of generality that $\bar
w_k=\tilde w_k$ on a neighborhood of $\partial (sQ') \times I$ and
\begin{equation}\label{eqc7}\frac{d\mu^c}{d|D_\a^c u|}(x_0)\geq \limsup_{k \to +\infty}
\int_{(sQ') \times I} W^\infty\left(\nabla_\a \bar w_k
\Big|\frac{1}{\d_k} \nabla_3 \bar w_k\right) dx.\end{equation} We
now compute
\begin{equation}\label{eqc9}D_\a z_k(sQ')=\frac{D_\a u(Q'(x_0,s
\rho_k))}{|D_\a u|(Q'(x_0,\rho_k))} \quad \text{ and } D_\a \tilde
w_k((sQ') \times I)=sA_k\end{equation} where
$$A_k:=a \otimes e_2 [(\theta * \varrho_k)(s/2) - (\theta * \varrho_k)(-s/2)].$$
Note that by (\ref{cantor2}), (\ref{cantor7}), (\ref{eqc8}) and
(\ref{eqc9}), we have that
\begin{eqnarray}\label{lastlast}
\limsup_{k \to +\infty}|s A_k-A(x_0)| & = &  \limsup_{k
\to +\infty}|D_\a \tilde w_k((sQ') \times I) -A(x_0)| \nonumber\\
& = & \limsup_{k \to +\infty}|D_\a z_k((sQ') \times I) -A(x_0)|\nonumber\\
& = & \limsup_{k \to +\infty} \left|\frac{D_\a
u(Q'(x_0,s\rho_k))}{|D_\a u|(Q'(x_0,\rho_k))} - A(x_0)\right|\nonumber\\
& \leq & \limsup_{k \to +\infty}\frac{|D_\a u| (Q'(x_0,\rho_k)
\setminus Q'(x_0,s\rho_k))}{|D_\a u|(Q'(x_0,\rho_k))}\nonumber\\
&& + \limsup_{k \to +\infty} \left|\frac{D_\a
u(Q'(x_0,\rho_k))}{|D_\a u|(Q'(x_0,\rho_k))} - A(x_0)\right|\leq
1-t^2.
\end{eqnarray}
We now define our last sequence
$$\varphi_k(x_\a,x_3):=\bar w_k(s\, x_\a,x_3)-s A_k  x_\a + \d_k x_3
\left(\frac{d\ovb}{d|D_\a^c u|}(x_0) - \frac{1}{\d_k} \int_{Q'
\times I} \nabla_3 \bar w_k(s\, y_\a,y_3)\, dy\right).$$ As $\bar
w_k=\tilde w_k$ on $\partial (sQ') \times I$ and $\tilde w_k$
depends only on $(x_2,x_3)$, it is clear from the definition of
$A_k$ that $\varphi_k$ is $1$-periodic in the direction $e_1$. A
simple computation shows that for a.e. $x_1$ and $x_3 \in I$, then
$\varphi_k(x_1,-1/2,x_3)=\varphi_k(x_1,1/2,x_3)$ so that $\varphi_k$
is also $1$-periodic in the $e_2$ direction. Moreover we have that
$$\frac{1}{\d_k} \int_{Q' \times I} \nabla_3\varphi_k(y)\, dy
=\frac{d\ovb}{d|D_\a^c u|}(x_0).$$ Hence using (\ref{eqc7}), the
Lipschitz condition (\ref{Winftylip}) satisfied by $W^\infty$ and a
change of variable, we obtain that
\begin{eqnarray*}
\frac{d\mu^c}{d|D_\a^c u|}(x_0) & \geq  & \limsup_{k \to +\infty}
s\int_{Q' \times I} W^\infty\left(A(x_0) + \nabla_\a \varphi_k
\Big|\frac{s}{\d_k} \nabla_3 \varphi_k\right) dx\\
&& - L s^2 \limsup_{k \to +\infty}\left|\frac{d\ovb}{d|D_\a^c
u|}(x_0) - \frac{1}{\d_k s^2} \int_{(sQ') \times I} \nabla_3 \bar
w_k(y)\, dy\right|- L s\limsup_{k \to +\infty}|s A_k-A(x_0)|.
\end{eqnarray*} But since $\lambda^c(\partial (sQ'))=0$, it follows that
$$\frac{1}{\d_k}\int_{(sQ') \times I}\nabla_3 \bar w_k\, dx \to
\frac{d\ovb }{d|D^c_\a u|}(x_0) \nu^c(sQ')$$ and from
(\ref{lastlast}) that
\begin{equation}\label{1022}
\frac{d\mu^c}{d|D_\a^c u|}(x_0) \geq s\mathcal
Q^*(W^\infty)\left(A(x_0)\Big|\frac{d\ovb}{d|D_\a^c u|}(x_0)\right)
- L\left|\frac{d\ovb}{d|D_\a^c u|}(x_0) \right| |s^2 - \nu^c(sQ')|-
Ls(1-t^2).\end{equation} From (\ref{cantor7}) we infer that
$\nu^c(sQ') \geq \nu^c(\overline{tQ'}) \geq t^2$ and thus
$$t^2 - s^2 \leq \nu^c(sQ') - s^2 \leq 1-s^2.$$
Passing to the limit first as $s \to 1^-$ and then as $t \to 1^-$,
we deduce from (\ref{1022}) that
$$\frac{d\mu^c}{d|D_\a^c u|}(x_0) \geq \mathcal
Q^*(W^\infty)\left(A(x_0)\Big|\frac{d\ovb}{d|D_\a^c
u|}(x_0)\right)$$ and (\ref{muc}) follows from Proposition
\ref{denssurf}.

\vskip5pt

\noindent {\bf Proof of (\ref{mus}).} Let $x_0 \in \o$ be such that
the Radon-Nikod\'ym derivative of $\mu$ and $\ovb$ at $x_0$ with
respect to $|\ovb^\sigma|$ exist and are finite, such that
\begin{equation}\label{singular1}
\frac{d|\mu-\mu^\sigma|}{d|\ovb^\sigma|}(x_0)=\frac{d|\ovb-\ovb^\sigma|}{d|\ovb^\sigma|}(x_0)=0,
\end{equation}
and such that
\begin{equation}\label{singular2}
\frac{d\LL^2}{d|\ovb^\sigma|}(x_0)= \frac{d|D_\a
u|}{d|\ovb^\sigma|}(x_0)=0.
\end{equation}
Note that since $|\ovb^\sigma|$ is singular with respect to $\LL^2$,
$|D_\a u|$, $|\mu-\mu^\sigma|$ and $|\ovb-\ovb^\sigma|$, it turns
out that $|\ovb^\sigma|$ almost every points $x_0$ in $\o$ satisfy
these properties.

Let $\{\rho_k\} \searrow 0^+$ be such that $\lambda(\partial
Q'(x_0,\rho_k))=\mu(\partial Q'(x_0,\rho_k))=0$ for each $k \in \Nb$
and extract eventually a subsequence (still denoted $\{\rho_k\}$)
such that there exists $\nu^\sigma \in {\rm Tan}(|\overline
b^\sigma|,x_0)$, {\it i.e.}, \begin{equation}\label{approxsigma}
\lim_{k \to +\infty} \frac{1}{|\overline b^\sigma|(Q'(x_0,\rho_k))}
\int_{\Rb^2}\varphi\left(\frac{x_\a-x_0}{\rho_k} \right)\,
d|\overline b^\sigma|(x_\a)=\int_{\Rb^2} \varphi(x_\a)\,
d\nu^\sigma(x_\a) \quad \text{ for all }\varphi \in \C_c(\Rb^2).
\end{equation}
Then by (\ref{singular1}) and a change of variables
\begin{eqnarray}\label{eqs1}
\frac{d\mu^\sigma}{d|\ovb^\sigma|}(x_0) & = &
\frac{d\mu}{d|\ovb^\sigma|}(x_0)= \lim_{k \to
+\infty} \frac{\mu(Q'(x_0,\rho_k))}{|\ovb^\sigma|(Q'(x_0,\rho_k))}\nonumber\\
& = &\lim_{k \to +\infty}\lim_{n \to +\infty} \frac{1}{t_k}\int_{Q'
\times I} W\left(\nabla_\a u_n(x_0+\rho_k
y_\a,y_3)\Big|\frac{1}{\e_n} \nabla_3 u_n(x_0+\rho_k y_\a,y_3)
\right)\, dy,
\end{eqnarray}
where
$$t_k:=\frac{|\ovb^\sigma|(Q'(x_0,\rho_k))}{\rho_k^2}.$$  Define
$$\left\{\begin{array}{l}
\ds
\psi_{n,k}(x_\a,x_3):=\frac{\rho_k}{|\ovb^\sigma|(Q'(x_0,\rho_k))}\left[
u_n (x_0+\rho_k x_\a,x_3) - \int_{Q' \times I}
u_n(x_0+\rho_ky_\a,y_3)\, dy \right],\\
\ds \psi_k(x_\a):=\frac{\rho_k}{|\ovb^\sigma|(Q'(x_0,\rho_k))}\left[
u (x_0+\rho_k x_\a) - \int_{Q'} u(x_0+\rho_ky_\a)\, dy_\a
\right].\end{array}\right.$$ Replacing in (\ref{eqs1}), we get that
\begin{equation}\label{eqs2}\frac{d\mu^\sigma}{d|\ovb^\sigma|}(x_0)=\lim_{k \to +\infty}\lim_{n \to
+\infty} \frac{1}{t_k}\int_{Q' \times I} W\left(t_k \left(\nabla_\a
\psi_{n,k} \Big|\frac{\rho_k}{\e_n} \nabla_3
\psi_{n,k}\right)\right) dx.\end{equation} Using the fact that $u_n
\to u$ in $L^1(\O;\Rb^3)$ we obtain that $\psi_{n,k} \to \psi_k$ in
$L^1(Q' \times I;\Rb^3)$ as $n \to +\infty$. Moreover, as
$\int_{Q'}\psi_k\, dx_\a =0$ and by (\ref{singular2}),
$$ |D_\a \psi_k|(Q')=
\frac{|D_\a u|(Q'(x_0,\rho_k))}{|\ovb^\sigma|(Q'(x_0,\rho_k))} \to
0,$$ the Poincar\'e-Wirtinger Inequality implies that $\psi_k \to 0$
in $L^1(Q';\Rb^3)$, hence
\begin{equation}\label{eqs3}
\lim_{k \to +\infty}\lim_{n \to +\infty}\|\psi_{n,k}\|_{L^1(Q'
\times I;\Rb^3)}=0.\end{equation}  Consider $\varphi \in
\C_0(Q';\Rb^3)$, then changing variables and using the fact that
$(1/\e_n)\int_I\nabla_3 u_n(\cdot,y_3)\, dy_3 \xrightharpoonup[]{*}
\ovb$ in $\M(\o;\Rb^3)$ together with (\ref{singular1}) and
(\ref{approxsigma}), it follows that
\begin{equation}\label{eqs4} \lim_{k \to +\infty}\lim_{n \to
+\infty}\int_{Q'}\varphi(x_\a) \cdot \left(\frac{\rho_k}{\e_n}\int_I
\nabla_3 \psi_{n,k}(x_\a,x_3)\, dx_3 \right) dx_\a =
\frac{d\ovb}{d|\ovb^\sigma|}(x_0) \cdot \int_{Q'}\varphi\,
d\nu^\sigma.
\end{equation}
Moreover since $\lambda(\partial Q'(x_0,\rho_k))=0$ for every $k \in
\Nb$, we deduce that
\begin{equation}\label{1734}
\lim_{k \to +\infty} \lim_{n \to +\infty} \frac{\rho_k}{\e_n}
\int_{Q'\times I} \nabla_3 \psi_{n,k}\, dy=
\frac{d\ovb}{d|\ovb^\sigma|}(x_0).
\end{equation}
Gathering (\ref{eqs2}), (\ref{eqs3}), (\ref{eqs4}) and (\ref{1734}),
and using the separability of $\C_0(Q';\Rb^3)$ together with a
standard diagonalization argument leads to the existence of a
subsequence $\{n_k\} \nearrow +\infty$ such that, setting
$\phi_k:=\psi_{n_k,k}$ and $\d_k:=\e_{n_k}/\rho_k$, then $\d_k
\searrow 0^+$, $\phi_k \to 0$ in $L^1(Q' \times I;\Rb^3)$,
$\frac{1}{\d_k} \int_I\nabla_3 \phi_k(\cdot,x_3)\, dx_3
\xrightharpoonup[]{*} \frac{d\ovb}{d|\ovb^\sigma|}(x_0) \nu^\sigma$
in $\M(Q';\Rb^3)$,
\begin{equation}\label{1736}
\frac{1}{\d_k}\int_{Q' \times I} \nabla_3 \phi_k\, dy \to
\frac{d\ovb}{d|\ovb^\sigma|}(x_0)
\end{equation} and
\begin{equation}\label{eqs5}\frac{d\mu^\sigma}{d|\ovb^\sigma|}(x_0)= \lim_{k \to +\infty}
\frac{1}{t_k}\int_{Q' \times I} W\left(t_k \left(\nabla_\a \phi_k
\Big|\frac{1}{\d_k} \nabla_3 \phi_k\right)\right) dx.\end{equation}
By virtue of the coercivity condition $(H_1)$, the sequence of
scaled gradients $\{(\nabla_\a \phi_k|(1/\d_k) \nabla_3 \phi_k)\}$
is uniformly bounded in $L^1(Q' \times I;\Rb^{3 \times 3})$. Thus
using hypothesis $(H_2)$ and H\"older's Inequality, we get that
\begin{eqnarray*}
&&\frac{1}{t_k}\int_{Q' \times I}\left|W^\infty\left(t_k
\left(\nabla_\a \phi_k \Big|\frac{1}{\d_k} \nabla_3
\phi_k\right)\right) - W\left(t_k \left(\nabla_\a \phi_k
\Big|\frac{1}{\d_k}
\nabla_3 \phi_k\right)\right)\right| dx\\
&&\hspace{2cm} \leq \frac{C}{t_k} +\frac{C}{t_k^r}\int_{Q' \times
I}\left|\left(\nabla_\a \phi_k \Big|\frac{1}{\d_k} \nabla_3 \phi_k\right)\right|^{1-r}\, dx\\
&&\hspace{2cm} \leq \frac{C}{t_k}
+\frac{C}{t_k^r}\left\|\left(\nabla_\a \phi_k| (1/\d_k) \nabla_3
\phi_k\right)\right\|_{L^1(Q' \times I;\Rb^{3 \times 3})}^{1-r} \to
0,
\end{eqnarray*}
where we used the fact that, thanks to (\ref{singular2}), $t_k  \to
+\infty$. But as $W^\infty$ is positively $1$-homogeneous, we get
from (\ref{eqs5}) that $$\frac{d\mu^\sigma}{d|\ovb^\sigma|}(x_0)=
\lim_{k \to +\infty} \int_{Q' \times I} W^\infty\left(\nabla_\a
\phi_k \Big|\frac{1}{\d_k} \nabla_3 \phi_k\right)\, dx.$$  Using
Remark \ref{bdry}, we can assume without loss of generality that
$\phi_k=0$ on a neighborhood of $\partial Q' \times I$. We now
define
$$\tilde \phi_k(x_\a,x_3):= \phi_k(x_\a,x_3) +\d_k x_3
\left(\frac{d\ovb}{d|\ovb^\sigma|}(x_0) - \frac{1}{\d_k} \int_{Q'
\times I}\nabla_3 \phi_k(y)\, dy\right).$$ Then, $\tilde \phi_k \in
W^{1,1}(Q' \times I;\Rb^3)$, $\tilde \phi_k(\cdot,x_3)$ is
$Q'$-periodic for $\LL^1$-a.e. $x_3 \in I$ and $$\frac{1}{\d_k}
\int_{Q' \times I} \nabla_3\tilde \phi_k\, dy =
\frac{d\ovb}{d|\ovb^\sigma|}(x_0).$$ Hence $\tilde \phi_k$ is an
admissible test function for $\mathcal
Q^*(W^\infty)\Big(0\big|\frac{d\ovb}{d|\ovb^\sigma|}(x_0)\Big)$ and
using the Lipschitz property (\ref{Winftylip}), we get that
$$\frac{d\mu^\sigma}{d|\ovb^\sigma|(x_0)} \geq \limsup_{k \to +\infty} \int_{Q' \times I}
W\left(\nabla_\a \tilde \phi_k\Big|\frac{1}{\d_k}  \nabla_3 \tilde
\phi_k\right) dy - L \limsup_{k \to +\infty}
\left|\frac{d\ovb}{d|\ovb^\sigma|}(x_0) - \frac{1}{\d_k} \int_{Q'
\times I}\nabla_3 \phi_k(y)\, dy\right|.$$  Finally, relation
(\ref{1736}) ensures that the last term in the previous inequality
is actually zero and thus, from Proposition \ref{denssurf},
$$\frac{d\mu^\sigma}{d|\ovb^\sigma|}(x_0) \geq
\mathcal Q^*(W^\infty)\left(0\;
\Big|\frac{d\ovb}{d|\ovb^\sigma|}(x_0)\right) = (\mathcal Q^*
W)^\infty \left(0\;\Big|\frac{d\ovb}{d|\ovb^\sigma|}(x_0)\right).$$

\section{The upper bound}

\begin{Lemma}\label{upperbound}
For any $(u,\ovb) \in BV(\o;\Rb^3) \times \M(\o;\Rb^3)$, then
$J_{\{\e_n\}}(u,\ovb,\o) \leq E(u,\ovb)$.
\end{Lemma}

\begin{proof}
Let $(u,\ovb) \in BV(\o;\Rb^3) \times \M(\o;\Rb^3)$. As in the proof
of the lower bound, according to the Besicovitch Decomposition
Theorem, we can decompose $\ovb$ into the sum of three mutually
singular measures $\ovb^a$, $\ovb^s$ and $\ovb^\sigma$ such that
$\ovb=\ovb^a+\ovb^s+\ovb^\sigma$ where $\ovb^a \ll \LL^2$, $\ovb^s
\ll |D^s_\a u|$.

{\bf Step 1. }Assume first that $\partial \o$ is Lipschitz. Then by
the locality result Lemma \ref{lemma2.1bfmbend}, it is enough to
check that
\begin{eqnarray}
\frac{d J_{\{\e_n\}}(u,\ovb,\cdot)}{d\LL^2}(x_0) \leq \mathcal
Q^*W\left(\nabla_\a u(x_0)\Big|\frac{d\ovb}{d\LL^2}(x_0)\right)
\quad \text{ for
$\LL^2$-a.e. }x_0 \in \o,\label{Ja}\\
\frac{d J_{\{\e_n\}}(u,\ovb,\cdot)}{d|D_\a^s u|}(x_0) \leq (\mathcal
Q^*W)^\infty\left(\frac{dD_\a u}{d|D_\a^s
u|}(x_0)\Big|\frac{d\ovb}{d|D_\a^s u|}(x_0)\right) \quad \text{ for
$|D_\a^s u|$-a.e. }x_0 \in \o,\label{Jjc}\\
\frac{d J_{\{\e_n\}}(u,\ovb,\cdot)}{d|\ovb^\sigma|}(x_0) \leq
(\mathcal
Q^*W)^\infty\left(0\Big|\frac{d\ovb}{d|\ovb^\sigma|}(x_0)\right)
\quad \text{ for $|\ovb^\sigma|$-a.e. }x_0 \in \o.\label{Js}
\end{eqnarray}
Indeed, if the three previous properties hold, since
$J_{\{\e_n\}}(u,\ovb,\cdot)$ is (the trace of) a Radon measure
absolutely continuous with respect to $\LL^2+|D_\a u| + |\ovb|$, we
deduce that
\begin{eqnarray*} J_{\{\e_n\}}(u,\ovb,\o) & \leq &\int_\o \mathcal Q^*W\left(\nabla_\a
u\Big|\frac{d\ovb}{d\LL^2}\right) dx + \int_{J_u}(\mathcal Q^\ast
W)^\infty\left((u^+-u^-)\otimes \nu_u,\frac{d\ovb}{d\HH^1\res\,
J_u}\right)d\HH^1\\
&&+\int_\o (\mathcal Q^*W)^\infty\left(\frac{dD_\a u}{d|D_\a^c
u|}\Big|\frac{d\ovb}{d|D_\a^c u|}\right)d|D_\a^c u|+
\int_\o(\mathcal
Q^*W)^\infty\left(0\Big|\frac{d\ovb}{d|\ovb^\sigma|}\right)d|\ovb^\sigma|,
\end{eqnarray*}
where we used the fact that $D_\a^s u = (u^+-u^-)\otimes \nu_u
\HH^1\res\, J_u + D_\a^c u$ and that $(\mathcal Q^*W)^\infty$ is
positively $1$-homogeneous. We postpone the proof of the three above
inequalities to the end of the section.\\

{\bf Step 2.} Let us now explain how to remove the Lipschitz
condition on $\partial \o$. As in the proof of Lemma
\ref{lemma2.1bfmbend}, for every $k \in \Nb$, it is possible to find
an increasing sequence of open sets $\o_k \subset\subset \o_{k+1}
\subset\subset \o$ such that $\partial \o_k$ is Lipschitz and
$|\ovb|(\partial \o_k)=0$ for each $k\in \Nb$. By Step 1 and Lemma
\ref{tracegamma}, there exists a sequence $\{u^k_n\} \subset
W^{1,1}(\o_k \times I;\Rb^3)$ such that $Tu^k_n=Tu$ on $\partial
\o_k \times I$, $u^k_n \to u$ in $L^1(\o_k \times I;\Rb^3)$,
$\frac{1}{\e_n}\int_I \nabla_3 u^k_n(\cdot,x_3)\, dx_3
\xrightharpoonup[]{*} \ovb$ in $\M(\o_k;\Rb^3)$ as $n \to +\infty$
and
\begin{equation}\label{omega'2}\limsup_{n \to +\infty}\int_{\o_k \times I} W\left(\nabla_\a
u^k_n \Big|\frac{1}{\e_n}\nabla_3 u^k_n\right) dx \leq
E(u,\ovb,\o_k)+\frac{1}{k} \leq
E(u,\ovb,\o)+\frac{1}{k}.\end{equation} We now apply (a slight
variant of) \cite[Lemma 2.5]{BouFonMas} to get a sequence $\{v^k_n\}
\subset W^{1,1}(\o \setminus \overline \o_k;\Rb^3)$ such that $v^k_n
\to u$ in $L^1(\o \setminus \overline \o_k;\Rb^3)$, $Tv^k_n=Tu$ on
$\partial \o_k$ and $|D_\a v^k_n|(\o \setminus \overline \o_k) \to
|D_\a u|(\o \setminus \overline \o_k)$ as $n \to +\infty$. Indeed an
inspection of the proof of \cite[Lemma 2.5]{BouFonMas} shows that,
since we do not prescribe the value of the trace on $\partial \o$,
it is not necessary to assume $\partial \o$ to be Lipschitz. Define
$w^k_n(x_\a,x_3):=u^k_n(x_\a,x_3) \chi_{\o_k}(x_\a) + v^k_n(x_\a)
\chi_{\o \setminus \overline \o_k}(x_\a)$. As $Tu^k_n=Tv^k_n=Tu$ on
$\partial \o_k \times I$, the sequence $w^k_n \in
W^{1,1}(\O;\Rb^3)$, $$\lim_{k \to +\infty}\lim_{n \to
+\infty}\|w^k_n - u\|_{L^1(\O;\Rb^3)}=0, \quad \lim_{k \to
+\infty}\lim_{n \to +\infty}|D_\a v^k_n|(\o \setminus \overline
\o_k)=0$$ and for any $\varphi \in \C_0(\o;\Rb^3)$, we have
$$\lim_{k \to +\infty}\lim_{n \to +\infty}\int_\o \varphi(x_\a)\cdot
\left( \frac{1}{\e_n}\int_I \nabla_3 w_n^k(x_\a,x_3)\, dx_3\right)
dx_\a = \int_\o \varphi(x_\a)\cdot d\ovb(x_\a).$$ Using the
separability of $\C_0(\o;\Rb^3)$ and a standard diagonalization
procedure, we obtain the existence of a sequence $k_n \nearrow
+\infty$ such that, setting $w_n:=w_n^{k_n}$, then $w_n \to u$ in
$L^1(\O;\Rb^3)$, $\frac{1}{\e_n} \int_I \nabla_3 w_n(\cdot,x_3)\,
dx_3 \xrightharpoonup[]{*} \ovb$ in $\M(\o;\Rb^3)$, $|D_\a
v^{k_n}_n|(\o \setminus \overline \o_{k_n}) \to 0$ and by
(\ref{omega'2}),
\begin{equation}\label{omega'3}\limsup_{n \to +\infty}\int_{\o_{k_n} \times I} W\left(\nabla_\a
u^{k_n}_n \Big|\frac{1}{\e_n}\nabla_3 u^{k_n}_n\right) dx \leq
E(u,\ovb,\o).\end{equation} Using the growth condition $(H_1)$
together with (\ref{omega'3}), we get that
$$J_{\{\e_n\}}(u,\ovb,\o) \leq \limsup_{n \to +\infty}\int_{\O} W\left(\nabla_\a w_n
\Big|\frac{1}{\e_n}\nabla_3 w_n\right) dx \leq E(u,\ovb,\o)$$ which
concludes the proof of the upper bound.
\end{proof}

\vskip5pt

\noindent {\bf Proof of (\ref{Ja}).} Fix a point $x_0\in \o$ such
that \begin{equation}\label{abs1} \frac{d\ovb}{d\LL^2}(x_0), \quad
\frac{dJ_{\{\e_n\}}(u,\ovb,\cdot)}{d {\cal L}^2}(x_0), \quad
\frac{dD_\a u}{d {\cal L}^2}(x_0)= \nabla_\alpha u(x_0)
\end{equation} exist and are finite, which is also a Lebesgue
point of $u$, $\nabla_\alpha u$ and $\frac{d\ovb}{d\LL^2}$, a point
of approximate differentiability for $u$, and such that
\begin{equation}\label{abs2}
\frac{d|D_\a^s u|}{d{\cal L}^2}(x_0)= \frac{d|\ovb -\ovb^a
|}{d\LL^2}(x_0)=0.
\end{equation}
Observe that since ${\cal L}^2$ is singular with respect to $|D_\a^s
u|$ and $|\ovb- \ovb^a|$, then ${\cal L}^2$-a.e. $x_0\in \o$ satisfy
all the above requirements.

Let $\{\rho_k\} \searrow 0^+$ be such that $|D_\a u|(\partial
Q'(x_0,\rho_k)) = |\ovb|(\partial Q'(x_0,\rho_k))=0$ for each $k \in
\Nb$. Let $\eta>0$ and consider $\lambda>0$ and $\varphi \in
W^{1,1}(Q' \times I;\Rb^3)$ such that $\varphi(\cdot,x_3)$ is
$Q'$-periodic for $\LL^1$-a.e. $x_3 \in I$, $\lambda
\int_I\nabla_3\varphi\, dy =\frac{d\ovb}{d\LL^2}(x_0)$ and
$$\int_Q W(\nabla_\alpha u(x_0)+\nabla_\alpha \varphi| \lambda
\nabla_3 \varphi)\, dx  \leq \mathcal Q^* W\left(\nabla_\a
u(x_0)\Big|\frac{d\ovb}{d\LL^2}(x_0)\right)+\eta.$$ Then, defining
$\varphi_n : \Rb^2 \times I \to \Rb^3$ by
\begin{equation}\label{varphih}
\varphi_n(x_\alpha,x_3):= \lambda\e_n
\varphi\left(\frac{x_\a}{\lambda\e_n}, x_3\right),
\end{equation}
it results that
\begin{equation}\label{varphiconvergences}
\left\{
\begin{array}{l}
\varphi_n \to 0 \hbox{ in }  L^1(Q'(x_0,\rho_k)\times I;\Rb^3),\\[0.2cm]
\frac{1}{\e_n} \int_I \nabla_3 \varphi_n(\cdot,x_3)\, dx_3
\xrightharpoonup[]{*} \frac{d\ovb}{d\LL^2}(x_0)\LL^2  \hbox{ in }
{\cal M}(Q'(x_0;\rho_k);\Rb^3).
\end{array}
\right.
\end{equation}
Let $\{\varrho_n\}$ be a sequence of standard symmetric mollifiers
chosen in such a way that
\begin{equation}\label{mollifier}
\lim_{n \to +\infty}\e_n \int_{Q'(x_0,\rho_k)}\big(|\ovb *
\varrho_n|+ |\nabla_\a (\ovb \ast \varrho_n)|\big) \, dx_\a=0
\end{equation}
and set $v_n(x_\a,x_3):= (u \ast \varrho_n)(x_\a) + \e_n x_3
(\ovb\ast \varrho_n)(x_\a)$. Define the sequence
\begin{equation}\label{recbulk}
w_n(x_\a,x_3):= v_n(x_\a,x_3) + \varphi_n(x_\a,x_3) - \e_n x_3
\frac{d\ovb}{d\LL^2}(x_0).
\end{equation}
It results from (\ref{varphih}), (\ref{varphiconvergences}),
(\ref{mollifier}), (\ref{recbulk}) and \cite[Theorem 2.2]{AFP} that
$$\left\{
\begin{array}{l}
w_n \to u  \hbox{ in } L^1(Q'(x_0,\rho_k)\times I;\Rb^3),\\[0.2cm]
\frac{1}{\e_n}\int_I  \nabla_3 w_n(\cdot,x_3)\, dx_3
\xrightharpoonup[]{*} \ovb \hbox{ in }  {\cal
M}(Q'(x_0,\rho_k);\Rb^3).
\end{array}\right.$$
Hence, taking $\{w_n\}$ as test function we get that
\begin{eqnarray*}
J_{\{\e_n\}}(u,\ovb,Q'(x_0,\rho_k)) &\leq & \liminf_{n\to
+\infty}\int_{Q'(x_0,\rho_k)\times I}W\left(\nabla_\alpha w_n\Big|
\frac{1}{\e_n} \nabla_3 w_n\right)dx\\
\\
&=&\liminf_{n \to +\infty}\int_{Q'(x_0,\rho_k)\times
I}W\left(\nabla_\a v_n+\nabla_\a \varphi_n\Big|\frac{1}{\e_n}
\nabla_3 v_n+\frac{1}{\e_n} \nabla_3
\varphi_n-\frac{d\ovb}{d\LL^2}(x_0)\right)dx
\end{eqnarray*}
and using the Lipschitz property (\ref{Wlip}) of $W$ together with
(\ref{varphih}), it follows that
\begin{eqnarray}\label{J-}
J_{\{\e_n\}}(u,\ovb,Q'(x_0,\rho_k)) &\leq &\liminf_{n \to
+\infty}\int_{Q'(x_0,\rho_k)\times I}W\left(\nabla_\a u(x_0)+
\nabla_\a \varphi \left(\frac{x_\a}{\lambda\e_n},x_3\right)\Big|
\lambda \nabla_3 \varphi\left(\frac{x_\a}{\lambda\e_n},x_3\right)\right)dx\nonumber\\
\nonumber\\
&&+ L\limsup_{n \to +\infty}\int_{Q'(x_0,\rho_k) \times
I}|\nabla_\alpha v_n -\nabla_\alpha u(x_0)|\, dx\nonumber\\
&&+ L\limsup_{n \to +\infty} \int_{Q'(x_0,\rho_k) \times
I}\left|\frac{1}{\e_n} \nabla_3 v_n -
\frac{d\ovb}{d\LL^2}(x_0)\right|\, dx.
\end{eqnarray}
Observe that $\nabla_\a v_n(x_\a,x_3)=(\nabla_\a u *
\varrho_n)(x_\a) + (D_\a^s u * \varrho_n)(x_\a) + \e_n x_3
\nabla_\alpha (\ovb \ast \varrho_n)(x_\a)$ hence,
\begin{eqnarray*}
\int_{Q'(x_0,\rho_k) \times I}|\nabla_\alpha v_n -\nabla_\alpha
u(x_0)|\, dx &  \leq & \int_{Q'(x_0,\rho_k)}|\nabla_\a u * \varrho_n
-\nabla_\alpha u(x_0)|\, dx_\a\\
&& + \int_{Q'(x_0,\rho_k)} \big( |D_\a^s u
* \varrho_n| + \e_n |\nabla_\a (\ovb \ast \varrho_n)| \big) \,
dx_\a.
\end{eqnarray*}
Thus, according to (\ref{mollifier}), \cite[Theorem 2.2]{AFP}, the
fact that $\nabla_\a u * \varrho_n \to \nabla_\a u$ in $L^1_{\rm
loc}(\o;\Rb^{3 \times 2})$ and that $|D^s_\a u|(\partial
Q'(x_0,\rho_k))=0$ for each $k \in \Nb$, we get that
\begin{eqnarray}\label{J-2}
\limsup_{n \to +\infty}\int_{Q'(x_0,\rho_k) \times I}|\nabla_\a v_n
-\nabla_\a u(x_0)|\, dx & \leq & \int_{Q'(x_0,\rho_k)}|\nabla_\a
u(x_\a) -\nabla_\a u(x_0)|\,
dx_\a\nonumber\\
&&+|D_\a^s u|(Q'(x_0,\rho_k)).
\end{eqnarray}

Similarly, since $(1/\e_n) \nabla_3 v_n = \ovb * \varrho_n$, it
implies that
\begin{eqnarray*}\int_{Q'(x_0,\rho_k) \times
I}\left|\frac{1}{\e_n} \nabla_3
v_n-\frac{d\ovb}{d\LL^2}(x_0)\right|\, dx & \leq &
\int_{Q'(x_0,\rho_k)}\left|\left(\frac{d\ovb}{d\LL^2} *
\varrho_n\right)(x_\a)-\frac{d\ovb}{d\LL^2}(x_0) \right|
dx_\a\\
&& + \int_{Q'(x_0,\rho_k)}|(\ovb - \ovb^a )* \varrho_n|(x_\a)\,
dx_\a.\end{eqnarray*} Since $|\ovb- \ovb^a|(\partial
Q'(x_0,\rho_k))=0$ for each $k \in \Nb$ and $\frac{d\ovb}{d\LL^2}
* \varrho_n \to \frac{d\ovb}{d\LL^2}$ in $L^1_{\rm loc}(\o;\Rb^3)$, it yields
\begin{eqnarray}\label{J-3}\limsup_{n \to +\infty}\int_{Q'(x_0,\rho_k) \times
I}\left|\frac{1}{\e_n} \nabla_3
v_n-\frac{d\ovb}{d\LL^2}(x_0)\right|\, dx & \leq &
\int_{Q'(x_0,\rho_k)}\left|\frac{d\ovb}{d\LL^2}(x_\a)-\frac{d\ovb}{d\LL^2}(x_0)\right|\, dx_\a\nonumber\\
&& + |\ovb - \ovb^a|(Q'(x_0,\rho_k)).\end{eqnarray} Gathering
(\ref{J-}), (\ref{J-2}) and (\ref{J-3}) and using the
Riemann-Lebesgue Lemma, we get that
\begin{eqnarray*}
J_{\{\e_n\}}(u,\ovb,Q'(x_0,\rho_k)) & \leq & \rho_k^2 \mathcal Q^*
W\left(\nabla_\a u(x_0)\Big|\frac{d\ovb}{d\LL^2}(x_0)\right)
+ \rho_k^2 \eta \\
&&+ L |D_\a^s u|(Q'(x_0,\rho_k))+ L |\ovb - \ovb^a|(Q'(x_0,\rho_k))\\
&& +L \int_{Q'(x_0,\rho_k)}|\nabla_\a u(x_\a) -\nabla_\a u(x_0)|\,
dx_\a\\
&& + L \int_{Q'(x_0,\rho_k)}\left|\frac{d\ovb}{d\LL^2}(x_\a) -
\frac{d\ovb}{d\LL^2}(x_0)\right| dx_\a.
\end{eqnarray*}
Now dividing the previous inequality by $\rho_k^2$, sending $k \to
+\infty$ and exploiting properties (\ref{abs1}) and (\ref{abs2}) of
the point $x_0$, it leads to
$$\frac{dJ_{\{\e_n\}}(u,\ovb,\cdot)}{d\LL^2}(x_0) \leq
\mathcal Q^* W\left(\nabla_\a
u(x_0)\Big|\frac{d\ovb}{d\LL^2}(x_0)\right) + \eta$$ and the
arbitrariness of $\eta$ gives the desired claim.

\vskip5pt

\noindent {\bf Proof of (\ref{Jjc}).} The proof develops in the same
spirit of that in \cite[Proposition 5.49]{AFP} (see also
\cite{ADM}). Let us introduce an auxiliary function $f: \Rb^{3
\times 3} \to [0,+\infty)$ defined by
$$f(\xi):=\sup_{t>0}\frac{\mathcal Q^* W(t\xi) - \mathcal
Q^*W(0)}{t}.$$ It turns out that $f$ is a positively $1$-homogeneous
continuous function. Moreover, by (\ref{Wlip}) there exists $L>0$
such that
\begin{equation}\label{growthlipfj}
f(\xi) \leq L |\xi| \quad\text{ and }\quad |f(\xi)-f(\xi')| \leq
L|\xi - \xi'|\quad \text{ for every }\xi, \, \xi' \in \Rb^{3 \times
3}.
\end{equation}
Using the growth properties of differential quotients of convex
functions, it is easily seen from Proposition \ref{convex} that if
$z$, $b \in \Rb^3$ and $\nu \in \mathbb S^1$, then $f(z \otimes
\nu|b)=(\mathcal Q^* W)^\infty(z \otimes \nu|b)$.

Fix a standard sequence of mollifiers $\{\varrho_j\}$. Then by
\cite[Theorem 2.2]{AFP}, we have that $(u*\varrho_j,\ovb *
\varrho_j) \in W^{1,1}(\o;\Rb^3) \times L^1(\o;\Rb^3)$, $u*\varrho_j
\to u$ in $L^1_{\rm loc}(\o;\Rb^3)$ and $\ovb * \varrho_j
\xrightharpoonup[]{*} \ovb$ in $\M_{\rm loc}(\o;\Rb^3)$.

Using the Besicovitch Decomposition Theorem we can write $(D_\a
u|\ovb) = (D_\a^s u|\ovb^s) + \lambda^s$ for some singular measure
$\lambda^s \in \M(\o;\Rb^{3 \times 3})$ with respect to $|D_\a^s
u|$. Consider $x_0 \in \o$ satisfying
\begin{equation}\label{mesetranj}
\frac{d\lambda^s}{d|D_\a^s u|}(x_0)= \frac{d\LL^2}{d|D_\a^s
u|}(x_0)=0,
\end{equation}
such that
\begin{equation}\label{mesjump}
\frac{dD_\a^s u}{d|D_\a^s u|}(x_0)=\frac{dD_\a u}{d|D_\a^s u|}(x_0)
\text{ is a rank one matrix}, \quad \frac{d\ovb^s}{d|D_\a^s
u|}(x_0)=\frac{d\ovb}{d|D_\a^s u|}(x_0).
\end{equation}
Assume further that $x_0$ is a  Lebesgue point of
\begin{equation}\label{lpj}f\left(\frac{dD_\a u}{d|D_\a^s
u|}\Big|\frac{d\ovb}{d|D_\a^s u|}\right)\end{equation} with respect
to $|D_\a^s u|$ and that
\begin{equation}\label{mesjump2}
\frac{dJ_{\{\e_n\}}(u,\ovb,\cdot)}{d|D_\a^s u|}(x_0)
\end{equation}
exists and is finite. Note that by Alberti's Rank One Theorem
\cite{A}, $|D_\a^s u|$ almost every points $x_0\in \o$ satisfy these
properties. Let $\{\rho_k\} \searrow 0^+$ be such that $|D^s_\a
u|(\partial Q'(x_0,\rho_k)) = |\lambda^s|(\partial
Q'(x_0,\rho_k))=0$ for every $k\in \Nb$.

By Remark \ref{sobolevcase} together with the sequential lower
semicontinuity of $J_{\{\e_n\}}$, we get that
\begin{eqnarray*}J_{\{\e_n\}}(u,\ovb,Q'(x_0,\rho_k)) & \leq & \liminf_{j \to +\infty}
J_{\{\e_n\}}(u*\varrho_j,\ovb*\varrho_j,Q'(x_0,\rho_k))\\
& = & \liminf_{j \to +\infty} \int_{Q'(x_0,\rho_k)} \mathcal Q^*W
\big(\nabla_\a (u*\varrho_j)|\ovb * \varrho_j \big)\,
dx_\a\\
& = & \liminf_{j \to +\infty} \int_{Q'(x_0,\rho_k)} \mathcal Q^*W
\big((D_\a u|\ovb) * \varrho_j \big)\, dx_\a,\end{eqnarray*} where
we used the fact that $\nabla_\a (u*\varrho_j) = (D_\a
u)*\varrho_j$. By definition of $f$, it follows that
$$J_{\{\e_n\}}(u,\ovb,Q'(x_0,\rho_k)) \leq \liminf_{j \to +\infty} \int_{Q'(x_0,\rho_k)}
f\big((D_\a u|\ovb) * \varrho_j \big)\, dx_\a + \mathcal Q^*
W(0)\rho_k^2$$ and using its Lipschitz property (\ref{growthlipfj}),
we get that
\begin{eqnarray*}
J_{\{\e_n\}}(u,\ovb,Q'(x_0,\rho_k)) & \leq & \liminf_{j \to +\infty}
\int_{Q'(x_0,\rho_k)} f\big((D_\a^s u|\ovb^s) * \varrho_j
\big)\, dx_\a + \mathcal Q^* W(0)\rho_k^2\\
&&+ L \limsup_{j \to +\infty}\int_{Q'(x_0,\rho_k)}
|\lambda^s*\varrho_j|\, dx_\a.
\end{eqnarray*}
Since $|\lambda^s|(\partial Q'(x_0,\rho_k))=0$ for each $k \in \Nb$,
then \cite[Theorem 2.2]{AFP} implies that
\begin{eqnarray*}
J_{\{\e_n\}}(u,\ovb,Q'(x_0,\rho_k)) & \leq & \liminf_{j \to +\infty}
\int_{Q'(x_0,\rho_k)} f\big((D_\a^s u|\ovb^s) * \varrho_j
\big)\, dx_\a + \mathcal Q^* W(0)\rho_k^2\\
&&+ L |\lambda^s|(Q'(x_0,\rho_k)).
\end{eqnarray*}
As $(D_\a^s u|\ovb^s) * \varrho_j \xrightharpoonup[]{*}(D_\a^s
u|\ovb^s)$ in $\M_{\rm loc}(\o;\Rb^{3 \times 3})$ as $j \to
+\infty$, in particular we have that
$$(D_\a^s u|\ovb^s) * \varrho_j \xrightharpoonup[j \to
+\infty]{*}(D_\a^s u|\ovb^s) \text{ in }\M(Q'(x_0,\rho_k);\Rb^{3
\times 3}).$$ Moreover as $|D^s_\a u|(\partial Q'(x_0,\rho_k))=0$,
is follows from \cite[Theorem 2.2]{AFP} that
$$\int_{Q'(x_0,\rho_k)}|(D_\a^s u|\ovb^s)*\varrho_j| \, dx_\a \xrightarrow[j \to +\infty]{}
|(D_\a^s u|\ovb^s)|(Q'(x_0,\rho_k)).$$  Hence, applying Reshetnyak's
Continuity Theorem (see {\it e.g.} \cite[Theorem 2.39]{AFP}), we
infer that
\begin{eqnarray*}
J_{\{\e_n\}}(u,\ovb,Q'(x_0,\rho_k)) & \leq & \int_{Q'(x_0,\rho_k)}
f\left(\frac{dD^s_\a u}{d|D_\a^s
u|}\Big|\frac{d\ovb^s}{d|D^s_\a u|}\right)d|D^s_\a u| + \mathcal Q^* W(0)\rho_k^2\\
&&+ L |\lambda^s|(Q'(x_0,\rho_k)),
\end{eqnarray*}
where we used the fact that $f$ is positively $1$-homogeneous and
that $(D_\a^s u|\ovb^s)$ is absolutely continuous with respect to
$|D_\a^s u|$. Dividing the previous inequality by $|D^s_\a
u|(Q'(x_0,\rho_k))$, sending $k \to +\infty$ and using
(\ref{mesetranj}), (\ref{mesjump}), (\ref{lpj}) and
(\ref{mesjump2}), we deduce that
$$\frac{dJ_{\{\e_n\}}(u,\ovb,\cdot)}{d|D^s_\a u|}(x_0) \leq f\left(\frac{dD_\a
u}{d|D_\a^s u|}(x_0)\Big|\frac{d\ovb}{d|D^s_\a
u|}(x_0)\right)=(\mathcal Q^* W)^\infty\left(\frac{dD_\a u}{d|D_\a^s
u|}(x_0)\Big|\frac{d\ovb}{d|D^s_\a u|}(x_0)\right)$$ since
$\frac{dD_\a u}{d|D_\a^s u|}(x_0)$ is a rank one matrix.

\vskip5pt

\noindent {\bf Proof of (\ref{Js}).} The proof for estimating from
above the term concerning the singular part is analogous to the
previous one.

Using the Besicovitch Decomposition Theorem we can write $(D_\a
u|\ovb) = (0|\ovb^\sigma) + \lambda^\sigma$ for some singular
measure $\lambda^\sigma \in \M(\o;\Rb^{3 \times 3})$ with respect to
$|\ovb^\sigma|$. Consider $x_0 \in \o$ to be a Lebesgue point of
\begin{equation}\label{lp}
f\left(0\, \Big|\frac{d\ovb^\sigma}{d|\ovb^\sigma|}\right)
\end{equation} with respect to
$|\ovb^\sigma|$ satisfying
\begin{equation}\label{meset}
\frac{d|\ovb-\ovb^\sigma|}{d|\ovb^\sigma|}(x_0)=\frac{d\lambda^\sigma}{d|\ovb^\sigma|}(x_0)=
\frac{d\LL^2}{d|\ovb^\sigma|}(x_0)=0,
\end{equation}
and such that
\begin{equation}\label{messing}
\frac{dJ_{\{\e_n\}}(u,\ovb,\cdot)}{d|\ovb^\sigma|}(x_0)
\end{equation}
exists and is finite. Note that $|\ovb^\sigma|$ almost every points
$x_0\in \o$ satisfy these properties. Let $\{\rho_k\} \searrow 0^+$
be such that $|\ovb^\sigma|(\partial Q'(x_0,\rho_k)) =
|\lambda^\sigma|(\partial Q'(x_0,\rho_k))=0$ for every $k\in \Nb$.

Arguing exactly as in the previous subsection, we obtain that
\begin{eqnarray*}
J_{\{\e_n\}}(u,\ovb,Q'(x_0,\rho_k)) & \leq & \int_{Q'(x_0,\rho_k)}
f\left(0\, \Big|\frac{d\ovb^\sigma}{d|\ovb^\sigma|}\right)d|\ovb^\sigma| + \mathcal Q^* W(0)\rho_k^2\\
&&+ L |\lambda^\sigma|(Q'(x_0,\rho_k)).
\end{eqnarray*}
Dividing the previous inequality by $|\ovb^\sigma|(Q'(x_0,\rho_k))$,
sending $k \to +\infty$ and using (\ref{lp}), (\ref{meset}) and
(\ref{messing}), it implies that
$$\frac{dJ_{\{\e_n\}}(u,\ovb,\cdot)}{d|\ovb^\sigma|}(x_0) \leq
f\left(0 \, \Big|\frac{d\ovb}{d|\ovb^\sigma|}(x_0)\right) =(\mathcal
Q^* W)^\infty\left(0\,
\Big|\frac{d\ovb}{d|\ovb^\sigma|}(x_0)\right).$$

\vskip10pt

\noindent {\bf Acknowledgements:} The authors wish to thank Irene
Fonseca for having drawn this problem to their attention. They also
gratefully acknowledge the anonymous referee for his careful reading
of the manuscript, and for his remarks which improved the original
version of the paper. J.-F.B. has been partially supported by the
MULTIMAT Marie Curie Research Training Network MRTN-CT-2004-505226
``Multi-scale modelling and characterisation for phase
transformations in advanced materials''. E.Z. aknowledges the
support of GNAMPA and of PRIN ``Modelli e algoritmi di
ottimizzazione per il progetto di reti wireless''.

\end{document}